\documentclass[reqno,11pt]{amsart}

\usepackage{amssymb,graphics,graphicx,wrapfig}

\def\eps{\varepsilon}

\def\be{\begin{equation}}
\def\ee{\end{equation}}
\def\ba{\begin{align}}
\def\bm{\begin{multline}}
\def\bfig{\begin{figure}[htb]}
\def\efig{\end{figure}}

\numberwithin{equation}{section}
\newtheorem{theorem}{Theorem}[section]

\newtheorem{lemma}[theorem]{Lemma}

\newtheorem{assumption}{Assumption}

\DeclareMathSymbol{\leqslant}{\mathalpha}{AMSa}{"36}
\DeclareMathSymbol{\geqslant}{\mathalpha}{AMSa}{"3E}
\DeclareMathSymbol{\doteqdot}{\mathalpha}{AMSa}{"2B}
\DeclareMathSymbol{\circlearrowright}{\mathalpha}{AMSa}{"08}
\DeclareMathSymbol{\subsetneq}{\mathalpha}{AMSb}{"28}
\DeclareMathSymbol{\supsetneq}{\mathalpha}{AMSb}{"29}
\renewcommand{\leq}{\;\leqslant\;}
\renewcommand{\geq}{\;\geqslant\;}

\newcommand{\dd}{{\rm d}}
\newcommand{\e}[1]{\,{\rm e}^{#1}\,}

\newcommand{\upchi}{\raise 2pt \hbox{$\chi$}}

\def\writefig#1 #2 #3 {\rlap{\kern #1 truecm \raise #2 truecm
\hbox{#3}}}

\newcommand{\caM}{{\mathcal M}}

\def\bbone{{\mathchoice {\rm 1\mskip-4mu l} {\rm 1\mskip-4mu l} {\rm 1\mskip-4.5mu l} {\rm 1\mskip-5mu l}}}

\newcommand{\bbE}{{\mathbb E}}

\newcommand{\bbN}{{\mathbb N}}

\newcommand{\bbP}{{\mathbb P}}

\newcommand{\bbR}{{\mathbb R}}

\newcommand{\bbT}{{\mathbb T}}

\newcommand{\bbZ}{{\mathbb Z}}

\usepackage{todonotes} 

\presetkeys{todonotes}{backgroundcolor=black!5, inline}{}

\usepackage[ngerman, english]{babel}
\usepackage[applemac]{inputenc}

\usepackage{tikz}
\usetikzlibrary{plotmarks}
\usepackage{multicol}
\usepackage{comment}
\usepackage{ mathrsfs }

\usepackage{amsthm}
\usepackage{amsfonts,amsmath}
\newtheorem{theo}[theorem]{Theorem} 
\newtheorem{lem}[theorem]{Lemma}
\newtheorem{bem}[theorem]{Remark}
\newtheorem{cor}[theorem]{Corollary}
\newtheorem{prop}[theorem]{Proposition}
\usepackage{calligra}
\usepackage{bibgerm}
\usepackage{url}
\usepackage{textcomp}
\usepackage{float}
\usepackage{caption}
\usepackage{subcaption}

\usepackage{hyperref}
\hypersetup{
    colorlinks=true, 
    linktoc=all,     
    linkcolor=blue,  
}

\newcommand{\T}{\mathbb{T}} 
\newcommand{\PP}{\mathbb{P}}

\newcommand{\id}{1\hspace{-0,9ex}1} 
\newcommand{\E}{\mbox{I\negthinspace E}}
\newcommand*\diff{\mathop{}\!\mathrm{d}}

\usepackage[multiple]{footmisc} 

\newcommand{\getmydate}{%
  \ifcase\month%
    \or Januar\or Februar\or M\"arz%
    \or April\or Mai\or Juni\or Juli%
    \or August\or September\or Oktober%
    \or November\or Dezember%
  \fi\ \number\year%
}

\begin{document}


\title{Hydrodynamic limits of finite range exclusion processes}

\author{Arne Berrens}
\address{Arne Berrens \hfill\newline
\indent FB Mathematik, TU Darmstadt} \hfill\newline
\email{berrens@mathematik.tu-darmstadt.de}

\author{Volker Betz}
\address{Volker Betz \hfill\newline
\indent FB Mathematik, TU Darmstadt} \hfill\newline
\email{betz@mathematik.tu-darmstadt.de}

\author{Jan Giesselmann}
\address{Jan Giesselmann \hfill\newline
\indent FB Mathematik, TU Darmstadt} \hfill\newline
\email{giesselmann@mathematik.tu-darmstadt.de}

\author{Stefan Walter}

\maketitle

\begin{quote}
{\small
{\bf Abstract.}
We consider variants of the simple symmetric exclusion process (SSEP) where particles are 'large'; 
while in the SSEP, each lattice point is either occupied or not occupied by a single particle, in our model 
each particle is represented by a fixed shape, and the dynamics is so that shapes cannot overlap. In one dimension, 
this reduces to a finite range interaction that prevents particles from coming closer to each other than a certain distance. 
In this case, we prove the existence of the hydrodynamic limit, and identify the relevant macroscopic PDE, which is a porous medium
equation. In two dimensions, we have no rigorous results, but instead rely on numerical simulations to approximate the diffusion matrix.
Under the assumption that the hydrodynamic limit exists in two dimensions, we derive properties of the diffusion matrix based on the shape of the particles.

}

\vspace{1mm}

\vspace{1mm}
\noindent
{\footnotesize {\it 2020 Math.\ Subj.\ Class.:} 60K35, 82C22, 65C35}
\end{quote}

\section{Introduction}\label{Intro}

The simple symmetric exclusion process (SSEP) is a fundamental interacting particle system, and 
its large scale behaviour is well understood. The process consists of a collection of particles performing 
continuous time, simple symmetric random walks on $\bbZ^d$, but interacting through a hard core exclusion condition: 
if a particle attempts a jump to a site that is already occupied, that jump is not made. The hydrodynamic limit consists of 
rescaling the lattice $\bbZ^d$ so that neighbouring lattice sites have a distance of $1/N$, at the same time observing the 
process for a time $tN^2$ with $t > 0$, and then letting $N \to \infty$. If the initial configuration of particles 
(interpreted as atomic measures with unit atoms at the position of the particles) weakly 
converges to a measure $u(x,0) \dd x$ with a smooth Lebesgue density $u$, then for any time $s \leq t$, the configuration 
of the process at time $s N^2$ also converges to a similar measure $u(x,s) \dd x$, and its Lebesgue density is the solution 
of the heat equation $\partial_t u(x,t) = \Delta u(x,t)$. Details of and rigorous conditions and proofs 
for this standard hydrodynamic limit can be found in Chapter 4 of \cite{Kipnis99}. 

A very natural extension of the SSEP is to give each particle a non-trivial shape $A \subset \bbZ^d$: a 
particle at position $x \in \bbZ^d$ actually 'occupies' all lattice points $x + A$, meaning that if some other particle 
attempts to jump to a position $y \in \bbZ^d$, the jump is only made if $(x+A) \cap (y+A) = \emptyset$. The SSEP is the special 
case where $A = \{0\}$, and in general it is sensible to assume that $A$ is a bounded set containing $0$. Such models are natural
candidates to describe systems of physical particles that are spatially extended, and/or lack rotational symmetry. Indeed, 
this model appears at several places in the literature. The one-dimensional model goes back at least to the work \cite{MacDonald68}, where the 
version with asymmetric jump rates is used to model biopolymerization along a nucleic-acid template. 
In \cite{Schoenherr04}, this model has been analyzed more extensively, in particular the diffusion of a tagged particle
and the hydrodynamic limit. 

For the two-dimensional model, particular attention has been paid to the stationary distribution of the 
generalized SSEP where $A$ is a square of side length two. The reason is that the stationary distribution of this model exhibits a phase transition: 
when the density (which is a conserved quantity of the dynamics) is below some critical value, there is a unique stationary measure
for the stochastic dynamics. In contrast, for densities higher than this value, there are precisely four different stationary 
distributions: one distinction is that the mobility in one of the coordinate directions is higher than in the 
other one. For each preferred coordinate direction, another distinction is that either most of the mobility is on rows (columns)
with even second (first) coordinate, or in those with odd second (first) coordinate. The heuristic reason for this phenomenon is not 
very hard to understand: at very high densities, to even achieve such a high packing particles should be mostly in neat rows 
or columns, because chaotic placement takes up too much space. Then, one can study the dynamics of holes instead of particles. 
But holes have a rectangular shape (2 times 1), and so in order for a particle to switch between two columns, two holes 
in one of these columns have to meet. This event happens on a slower time scale than the movement of the holes themselves. 
While numerical studies confirming the phase transition go back to at least \cite{fernandes2007}, a mathematical proof 
has only been given very recently in the important paper of Hadas and Peled \cite{HadasPeled2026}.

The purpose of the present paper is an investigation of the dynamics, in particular the hydrodynamic limit, of the
generalized SSEP in one and two dimensions. In one dimension, we give a complete, rigorous proof of the convergence of the 
empirical density to the solution of a certain porous media equation. The equation is the one obtained
in \cite{Schoenherr04} in a non-rigorous fashion, and we thus confirm those results. The proof itself follows the relatively
well established route laid out in e.g.\ \cite{Kipnis99}, although confirming that the various steps go through still requires 
some work. The two-dimensional model is much harder. First of all, it is reasonable to avoid the high density phase where multiple
stationary distributions exist. But even in the subcritical regime, the model is not of gradient type.
The gradient condition requires that the jump rates can be expressed using discrete gradients of a local function (see \cite[Definition 2.5]{Kipnis99} for a precise definition).
This missing structure makes it extremely
challenging both from a theoretical point of view, as well as for guessing the limiting PDE in the hydrodynamic limit. The 
only prediction for the dynamics of such a model that we are aware of can be found in \cite{Bio11} for the case $A = \{0\} \times \{0,1,\ldots,r\}$, 
but in the reduction to one dimension, their predicted diffusivity differs from the one in \cite{Schoenherr04}, and therefore it can be suspected that also the prediction for two dimensions is not correct.

For the purpose of this paper, we assume without proof the existence of the hydrodynamic limit, and that the limiting PDE is 
a porous media equation.
Under these conditions we show that certain symmetries of the exclusion zone $A$ are inherited by the diffusion matrix.
More precisely, we prove that the diffusion matrix is diagonal if $A$ is rectangular and that its diagonal entries coincide if $A$ is square.

Additionally, we conducted careful numerical experiments. 
We observe that at low densities the diffusivity is bounded away from zero and increases with the density, similar to the one dimensional case.
However, it reaches its maximum at an intermediate density and subsequently decreases up to the critical density. This differs
from the qualitative behaviour of the solution found in \cite{Bio11}. At the end of our paper, we discuss the method of \cite{Bio11} and 
point out where, in our opinion, it fails to catch the correct dynamic behaviour of the model. 

The plan of the paper is as follows: in Section \ref{sec:model}, we introduce the model and present our rigorous, one-dimensional
results. Section \ref{sec:stationary distribution} derives the stationary distribution of the generalized SSEP in one dimension,
an important ingredient for the proof. Section \ref{sec:proof} contains the proof of the hydrodynamic limit in one dimension.
Section \ref{sec:2dim} deals with the two-dimensional model, where we present our numerical experiments and their results, compare them to the findings of \cite{Bio11}, 
and discuss likely sources of the discrepancy. 

\section{Model and results} \label{sec:model}

Let $\bbT_N^d = \bbZ^d \cap [1,N]^d$, and $\Omega_N := \{0,1\}^{\bbT_N^d}$. A {\em particle configuration} is an element 
$\eta = (\eta(x))_{x \in \bbT_N^d} \in \Omega_N$, where $\eta(x) = 1$ is interpreted as a particle being present at $x$, 
and $\eta(x)=0$ as no particle being present. Let $A$ be a finite subset of 
$\bbZ^d$, and set
\[
\Omega_{N,A} = \{ \eta \in \Omega_N: \eta(x) \eta(y) \bbone_{\{(x+ A) \cap (y + A) \neq \emptyset\}} = 0 \text{ for all } x,y 
\in \bbT_N^d, x \neq y\}
\]
Addition is understood with periodic boundary conditions, i.e.\ modulo $N$ in each coordinate direction. In other words, 
we interpret $\T^d_N$ as the discrete, $d$-dimensional torus. This will avoid boundary effects later on. 

When $A$ is interpreted as the shape of a particle at $0$, then $x + A$ is the shape of a particle at $x$, and thus 
$\Omega_{N,A}$ is the set of all particle configurations where no two particles overlap. $A = \{0\}$ imposes no extra constraints, 
while for larger $A$, $\Omega_{N,A}$ is a proper subset of $\Omega_N$. We will write $\eta^N$ for the elements of $\Omega_{N,A}$. 

The {\em simple symmetric exclusion process with excluded volume $A$ (SSEP($A$))} is the continuous time Markov process with state space 
$\Omega_{N,A}$ and generator 
\begin{align}\label{eqn:gen_A}
	[L_N f](\eta^N) = \frac12 \sum_{x \in \bbT_N^d} \sum_{y \sim x} w_A(\eta^N,x,y) (f([\eta^N]^{x,y}) - f(\eta)), 
\end{align} 
with
\[
w_A(\eta^N, x,y) = \eta^N(x) \prod_{z \neq x} \Big( 1 - \eta^N(z) 
	\bbone_{(y + A) \cap (z+A) \neq \emptyset} \Big), 
\]
acting on the bounded functions from $\Omega_{N,A}$ to $\bbR$. Here, as usual, $[\eta^N]^{x,y}$ 
is the configuration where the values at $x$ and $y$ are swapped: $[\eta^N]^{x,y}(x) = \eta^N (y)$, 
$[\eta^N]^{x,y}(y) = \eta^N (x)$, and $[\eta^N]^{x,y}(z) = \eta^N (z)$ for all other $z$. 
The relation $y \sim x$ means that $x$ and $y$ are neighbours in $\T_N^d$, where again we use periodic boundary conditions. 

Notice that $w_A(\eta^N,x,y) = 0$ if $\eta^N(x) = 0$ or if $[\eta^N]^{x,y} \notin \Omega_{N,A}$, otherwise it is $1$. This means that 
particles jump to neighbouring sites with rate $1$ if and only if the overlap exclusion condition is satisfied for the new configuration. 
In particular, the dynamics indeed leaves $\Omega_{N,A}$ invariant. Given an initial condition
$\bar \eta^N \in \Omega_{N,A}$, we have 
\[
\bbE(f(\eta^N_t) | \eta_0 = \bar \eta^N) = (P_t^N f)(\bar \eta^N), \qquad \text{where } \qquad P_t^N = \e{t L_N}. 
\]
A good introduction into these types of processes can be found in \cite{Liggett10}. 
An equivalent discrete time dynamics, which will be used for numerical simulations, is 
that a particle is selected uniformly at random, and it is proposed to move it to a uniformly chosen neighbouring lattice site. If the proposed
configuration satisfies the exclusion rule, the move is made, otherwise it is not made. Then, the next particle is selected, and so on.

Let $\bbT^d = [0,1]^d$. The {\em empirical measure} of a configuration $\eta^N \in \Omega_{N,A}$ is the finite, atomic measure  
on $\bbT^d$ given by 
\begin{align}
\label{eq:empirical measure}
\pi_{\eta^N}^{N} (du):= \frac{1}{N^d} \sum\limits_{x \in \T_N^d} \eta^N (x) \delta_{\frac{x}{N}} (\diff u),
\end{align}
where $\delta_y$ is the unit point mass at $y \in \bbT^d$.
Write for short $\pi_\eta^N=\pi_{\eta^N}^N$. Thus,  
for  continuous $G: \T^d \to \bbR$,
\begin{align}\label{summeempirisches1} 
	\langle \pi_{\eta}^{N}, G \rangle := \int\limits_{\T^d} G(u) \pi_{\eta}^{N} (\diff u) = \frac{1}{N^d} \sum\limits_{x \in \T_N^d} G(\tfrac{x}{N}) \eta^N (x) . 
\end{align}
The hydrodynamic limit studies the sequence of the time-dependent, random measures $\pi_{\eta_t}^N$ as $N \to \infty$, 
where $(\eta_t^N)$ is the SSEP($A$) on $\Omega_{N,A}$.  Some care is needed
with the topologies: for the set $\caM_+(\T^d)$ of positive measures on $\T^d$, we use the weak topology: $(\nu_n)$ converges 
weakly to some measure $\nu$ if for all bounded, continuous real functions $G$ on $\T$, we have 
\[
\lim_{n \to \infty} \langle \nu_n, G \rangle = \langle \nu, G \rangle.
\] 
This topology is metrizable: Choosing an $\| . \|_\infty$-dense subset $(G_k)$ of the bounded, 
continuous functions on $\T$ and defining 
\begin{equation} \label{eq:metric}
d(\nu, \tilde \nu) = \sum_{k \in \bbN} 2^{-k} \frac{|\langle \nu, G_k \rangle - \langle \tilde \nu, G_k \rangle|}
{1 + |\langle \nu, G_k \rangle - \langle \tilde \nu, G_k \rangle|}
\end{equation}
makes $\caM_+(\T^d)$ a complete metric space. We write $D([0,T], \caM_+(\T^d))$ for the set of functions from $[0,T]$ to 
$\caM_+(\T^d)$ that are right continuous with left limits, and use the Skorohod topology, see e.g. \cite{billingsley}. 
With this, $D([0,T], \caM_+(\T^d))$ is a complete, separable metric space, and $(t \mapsto \pi_{\eta_t}^N)_{N \in \bbN}$ 
is a sequence of $D([0,T], \caM_+(\T^d))$-valued random variables. For these random variables, we will again use the topology of weak 
convergence, but this time with respect to the distribution of $(\eta^N_t)_{t \leq T}$. 

For a complete description of the processes $(\eta_t^N)$, 
we still need to specify the sequence of initial conditions $(\eta_0^N)$. We specialise to the case where the 
resulting sequence $(\pi_{\eta_0}^N)$ of empirical measures converges to a non-random limiting measure, and in addition assume
that this measure has a Lebesgue density. Since for sequences of random variables with deterministic limits, convergence in probability
and weak convergence are equivalent, the following criterion is formulated in the former, which has technical advantages.  
Let $\rho_0$ be an integrable function on $\T$, and for each $N$ let $\eta_0^N$ be a random element of $\Omega_{N,A}$. 
We say that the sequence of random measures $(\pi^N_{\eta_0})$ is {\em associated to the profile} $\rho_0$ 
if it converges to the deterministic measure $\rho_0(x) \dd x$ in probability, i.e.\ if 
\[
\lim_{N \to \infty} \bbP \Big( \Big| \langle \pi^N_{\eta_0}, G \rangle - \int_\T \rho_0(x) G(x) \dd x \Big| > \eps \Big) = 0
\]
for all $\eps > 0$ and all bounded, continuous functions $G$. 

We can now formulate exactly the type of convergence we seek. Let $(\eta_0^N)_{N \in \bbN}$ be a sequence
of random elements of $\Omega_{N,A}$ such that the induced empirical measures $(\pi^N_{\eta_0})$ are associated to a profile $\rho_0$. 
Let $(\eta^N_t)_{t \geq 0, N \in \bbN}$ be the sequence of stochastic processes with initial conditions $(\eta_0^N)_{N \in \bbN}$
and generators given in \eqref{eqn:gen_A}, and let $((\pi^N_{N^2 t})_{t \in [0,T]})_{N \in \bbN}$ be the sequence of induced empirical measures, 
viewed as random elements of $D([0,T], \caM_+(\T^d))$. 
We say that {\em the hydrodynamic limit} of this sequence of processes exists if this sequence  
converges weakly (with respect to the distribution of $(\eta_t)_{t \leq T}$) to some (possibly random) element of 
$D([0,T], \caM_+(\T^d))$. \\[2mm]
\begin{bem}
	\begin{enumerate}
		\item[(a)] Note that time is scaled by a factor of $N^2$ in the definition of the hydrodynamic limit. This is the usual diffusive scaling, 
which is necessary to obtain non-trivial limits for symmetric particle systems.
		\item[(b)] In many important cases, the weak limit $(\pi_t)_{t \in [0,T]}$ of the sequence $((\pi_{\eta_{N^2 t}}^N)_{t \in [0,T]})_{N \in \bbN}$, has a Lebesgue density $\rho_t$ for each $\pi_t$. This gives rise
to a function $\rho: \bbT^d \times \bbR \to \bbR$, $(t,x) \mapsto \rho_t(x)$, which describes the evolution of the particle density 
in the limit on infinitely many particles.
	\end{enumerate}
\end{bem}

We now specialise to the one-dimensional case $d=1$. The overlap rule implies that the only interesting
sets $A$ are of the form $\{0, \ldots, r-1\}$ with $r \geq 1$. The convention is such that $r$ is the diameter of the particle, so 
$r=1$ leads to the SSEP. 
We lighten the notation by writing $\bbT_N = \bbT_N^1$, and $\bbT = [0,1] = \bbT^1$, and
$\Omega_{N,r} = \Omega_{N, \{1, \ldots, r\}}$
Also, we write $\eta_t$ instead of $\eta_t^N$ whenever $N$ is not explicitly needed. 
 The generator then becomes 
\begin{align} \label{gen} \begin{split}
[L_N f](\eta) = \frac12 \sum\limits_{\substack{x,y \in \T_N, \\ |x-y|=1}} \eta(x) \prod\limits_{k=0}^{r-1} (1-\eta (y+(y-x)k)) 
\big( f(\eta^{x,y})-f(\eta) \big). 
\end{split}
\end{align} 

\begin{theo} \label{thm:main}
Let $\rho_0$ an integrable function on $\T$, $r \geq 1$, for each $N \in \bbN$ let $\eta_0^N$ be an $\Omega_{N,r}$-valued 
random variable, and let $(\eta_t)$ be the Markov process with generator \eqref{gen} and initial distribution $\eta_0^N$. 
Assume that the sequence $(\pi_{\eta_0}^N)$ is associated to the profile $\rho_0$. Then the hydrodynamic limit of the 
empirical densities of $(\eta_t)$ exists, is deterministic, and is given by $t \mapsto \rho(t,x) \dd x$, where $\rho$ is 
the unique solution of the porous media equation
\begin{align} \label{PDE}
\begin{cases}
\partial_t \rho &= \frac{1}{2 } \partial_{x} (D(\rho) \partial_x \rho), \\  \rho(0,\cdot) &= \rho_0 (\cdot) ,
\end{cases}
\end{align}
where 
\[
D(\rho) = \frac{1}{(1 - (r-1) \rho)^2}.
\]
\end{theo}

Note that the diffusivity increases as the density increases towards its maximum value at $1/r$. On a qualitative level, 
this can actually be understood intuitively: $D(\rho)$ is the factor that determines how much  a given 
density slope $\partial_x \rho(x)$ drives particles towards lower densities. 
For low densities, the fact that the particles on the one side of $x$ are a bit 
more densely packed than on the other side only has a small effect on the $r$-exclusion rule; thus in the limit $\rho \to 0$, 
the diffusion constant is the same as for the SSEP. For high densities, however, unequal densities on the two sides of $x$ have 
an effect that is stronger than for the SSEP, due to more excluded volume per particle. So the tendency for particles to travel
from high density regions to low density regions increases. 

This is in contrast with our numerical findings for the two-dimensional model: when we set  
\begin{equation} \label{eq:A 2dim}
A = \{0, \ldots, r_1-1\} \times \{0, \ldots, r_2-1\},
\end{equation}
we find numerically that the diffusion constant first increases similar to the one-dimensional case, but later decreases again. 
See Figure \ref{fig:results} in Section \ref{sec:2dim}, where we will also discuss why this behaviour can be expected. 
As soon as $A \neq \{0\}$, 
the dynamics is no longer of gradient type in two dimensions, which makes the problem of the hydrodynamic limit 
much harder. While significant recent progress has been made by Funaki, Gu and Wang \cite{Funaki2026} for non-gradient exclusion processes, 
their methods require a 'soft' modification of the jump rates due to the presence of other (non nearest-neighbour) particles, in particular 
they need to be bounded away from zero, the only hard core constraint being the one already present in the SSEP, i.e.\ no two particles per site. 
So the SSEP($A$) is not covered. Moreover, the presence of the phase transition mentioned in the introduction destroys the uniqueness of the 
stationary distribution for densities above the critical one, adding further difficulties. We therefore restrict our attention to the 
subcritical regime. Since we cannot prove the existence of the hydrodynamic limit, we have to assume it. Given this, we 
prove in Section \ref{sec:2dim} that the diffusion matrix must be diagonal if $A$ is of the form \eqref{eq:A 2dim}, and numerically 
estimate its diagonal elements. In order not to overburden this section, we refer to Section \ref{sec:2dim} for further details.

\section{The stationary distribution in 1D}\label{sec:stationary distribution}
In this section and the next section, we will use the term $RrEP$ (Range $r$ exclusion process) for the SSEP($A$) with $A = \{0, \ldots, r-1\}$.
Here, we construct the stationary distribution of the RrEP. This will be a crucial ingredient for our proof in Section \ref{sec:proof}.

As the number of particles is conserved under the time evolution, the state space $\Omega_N$ can be divided into non-communicating subspaces $\Omega_{N,K} := \{ \eta \in \Omega_N \ : \ \sum_{x \in \T_N} \eta (x) = K \}$, each being irreducible due to the assumed periodic boundary conditions. Hence, one gets a unique stationary measure for each $K \in \{0,\dots,\lfloor \frac{N}{r} \rfloor\}$ and a whole family of stationary measures (called {\em canonical measures}) $(\mu_{N,K})_{K \in \{0,\dots,\lfloor \frac{N}{r} \rfloor \}}$ on $\Omega_N$. 
\begin{lem}
For $K \in \{0,\dots,\lfloor \frac{N}{r} \rfloor \}$, the uniform distribution $\mu_{N,K}$ on $\Omega_{N,K}$ is invariant for the RrEP.
\end{lem}
\begin{proof}
It is enough to show $\int L_N f \diff \mu_{N,K} = 0$ for all bounded $f$. Given a configuration $\eta \in \Omega_{N,K}$ with $\eta(x)=1$ it follows that $\eta^{x,x+1}(x+1) \cdot (1-\eta^{x,x+1}(x-r+1))=1$ for all $x \in \T_N$, as well as similarly for jumps to the other side. Thus
\begin{align*}
&\int (L_N f)(\eta) \diff \mu_{N,K} (\eta) \\ = & \sum\limits_{\eta \in \Omega_{N,K}} \mu_{N,K} (\eta) \frac12 \sum\limits_{\substack{x,y \in \T_N, \\ |x-y|=1}} \eta (x) (1-\eta(y+(y-x)(r-1))) \left( f(\eta^{x,y}) - f (\eta) \right)  \\ = &\frac{1}{2 |\Omega_{N,K}|} \sum\limits_{\substack{x,y \in \T_N, \\ |x-y|=1}}  \sum\limits_{\eta \in \Omega_{N,K}} \eta (x) (1-\eta(y+(y-x)(r-1))) \left( f(\eta^{x,y}) - f (\eta) \right) = 0 ,\end{align*}
due to \begin{align*} &\eta (x) (1-\eta(y+(y-x)(r-1))) ( f(\eta^{x,y})-f(\eta)) \\ = &- \eta^{x,y}(y) (1-\eta^{x,y}(x+(x-y)(r-1))) (f(\eta^{x,y})-f(\eta))\end{align*} whenever $\eta,\eta^{x,y} \in \Omega_{N,K}$ for $|x-y|=1$.
\end{proof} 
In the following we will focus only on $K \geq 1$ and conveniently think of the measures $\mu_{N,K}$ as being defined on $\Omega_{N}$, with 
$\mu_{N,K} ( \Omega_{N,K}) = 1$. 

Before stating the next lemma regarding the probability of certain events under $\mu_{N,K}$, it is important to know how many configurations are possible in total. In order to find $|\Omega_{N,K}|$ we will identify the connection to the range-1-exclusion case. Due to the periodic lattice, we have to condition first on $r-1$ neighbouring sites for some fixed $x\in\bbT_N$: \begin{align*} | \Omega_{N,K} | =  &| \underbrace{\{ \eta \in \Omega_{N,K}  :  \eta (x+k) = 0 \ \forall k \in \{0,\dots,r-2\} \}}_{=: \Omega_{N,K}^{0}}| \\ &+ |\underbrace{\{ \eta \in \Omega_{N,K}  :  \eta (x) = 1 \}}_{=: \Omega_{N,K}^{1,x}} |  + \dots + |\underbrace{\{ \eta \in \Omega_{N,K}  :  \eta (x+(r-2)) = 1 \}}_{=: \Omega_{N,K}^{1,x+(r-2)}} |  . \end{align*} 
Starting with the condition $\eta (x+k) = 0$ for all $k \in \{0,\dots,r-2\}$, one gets a bijection $\chi : \Omega_{N,K}^{0} \to \big\{ \eta \in \{0,1\}^{{N-(r-1)K}} \ : \ \sum_{j=1}^{N-(r-1)K}  \eta (j) = K  \big\}$ by removing the first $r-1$ rightward neighbouring sites to each particle (see Figure \ref{exampp} for an example with $r=2$). 
Since there are $K$ particles, the image space consists of $N-(r-1)K$ sites and particles might very well be neighbours, such that we obtained a suitable state space for the R1EP. Its cardinality is given by \[ | \Omega_{N,K}^{0}| = {N-(r-1)K \choose K} . \]
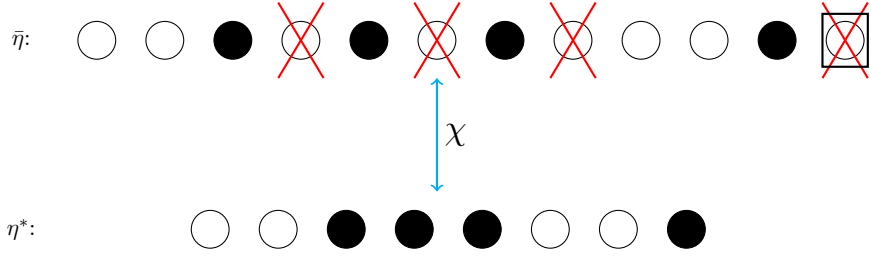
\begin{figure}[t]
\centering
\begin{tikzpicture} 

\node[scale=0.8] (a) at (-1,0) {$\bar{\eta}$:};

\node (A) at (0,0) [circle,fill=white,draw,minimum size=0.5cm] {};
\node (B) at (0.9,0) [circle,fill=white,draw,minimum size=0.5cm] {};
\node (C) at (1.8,0) [circle,fill=black,draw,minimum size=0.5cm] {};
\node (D) at (2.7,0) [circle,fill=white,draw,minimum size=0.5cm] {};
\node (E) at (3.6,0) [circle,fill=black,draw,minimum size=0.5cm] {};
\node (F) at (4.5,0) [circle,fill=white,draw,minimum size=0.5cm] {};
\node (G) at (5.4,0) [circle,fill=black,draw,minimum size=0.5cm] {};
\node (H) at (6.3,0) [circle,fill=white,draw,minimum size=0.5cm] {};
\node (I) at (7.2,0) [circle,fill=white,draw,minimum size=0.5cm] {};
\node (J) at (8.1,0) [circle,fill=white,draw,minimum size=0.5cm] {};
\node (K) at (9,0) [circle,fill=black,draw,minimum size=0.5cm] {};
\node (K) at (9.9,0) [circle,fill=white,draw,minimum size=0.5cm] {};

\draw[<->,thick,cyan] (4.5,-0.5) -- (4.5,-2);

\draw[thick,red] (2.4,-0.5) -- (3,0.5);
\draw[thick,red] (4.2,-0.5) -- (4.8,0.5);
\draw[thick,red] (6,-0.5) -- (6.6,0.5);
\draw[thick,red] (2.4,0.5) -- (3,-0.5);
\draw[thick,red] (4.2,0.5) -- (4.8,-0.5);
\draw[thick,red] (6,0.5) -- (6.6,-0.5);
\draw[thick,red] (9.6,0.5) -- (10.2,-0.5);
\draw[thick,red] (9.6,-0.5) -- (10.2,0.5);

\node[scale=0.8] (a) at (-1,-2.5) {$\eta^{*}$:};
\node[scale=1.2] (a) at (4.75,-1.25) {$\chi$};

\node (A) at (1.5,-2.5) [circle,fill=white,draw,minimum size=0.5cm] {};
\node (B) at (2.4,-2.5) [circle,fill=white,draw,minimum size=0.5cm] {};
\node (C) at (3.3,-2.5) [circle,fill=black,draw,minimum size=0.5cm] {};
\node (D) at (4.2,-2.5) [circle,fill=black,draw,minimum size=0.5cm] {};
\node (E) at (5.1,-2.5) [circle,fill=black,draw,minimum size=0.5cm] {};
\node (F) at (6,-2.5) [circle,fill=white,draw,minimum size=0.5cm] {};
\node (G) at (6.9,-2.5) [circle,fill=white,draw,minimum size=0.5cm] {};
\node (H) at (7.8,-2.5) [circle,fill=black,draw,minimum size=0.5cm] {};
\draw [thick] (9.6,-0.35) rectangle (10.2,0.35);

\end{tikzpicture}

\caption{Example of two corresponding configurations $\eta$ for the R2EP and $\eta^{*}$ for the R1EP under the map $\chi$ which removes the rightward (vacant) neighbouring site to each occupied site; a black circle symbolises a particle at this site, a white circle stands for a vacant site; the rectangle around the rightmost site of $\bar{\eta}$ reflects the conditioning on $\{ \eta \in \Omega_{12,4} \ | \ \eta (12) = 0 \}$.}
\label{exampp}
\end{figure}
Regarding the other conditions, note at first that $|\Omega_{N,K}^{1,x}| = \dots = | \Omega_{N,K}^{1,x+(r-2)}|$ once again due to periodicity. 
Assuming that $\eta(x)=1$, the matter is slightly different compared to the above case. There is still a bijection between the state space for the RrEP and an appropriate state space for the R1EP, but one has to take into account that we have to distribute $K-1$ particles instead of $K$ (as one is already on site $x$), and that there is one less site at disposal. Hence, \begin{align} |\Omega_{N,K}^{1,x} | = {N-(r-1)K-1 \choose K-1}, \label{wurst} \end{align} giving in total \begin{align} | \Omega_{N,K} | &= |\Omega_{N,K}^{0} | + (r-1) |\Omega_{N,K}^{1,x}| \notag \\ &=  {N-(r-1)K \choose K} + (r-1) {N-(r-1)K-1 \choose K-1} . \label{Kardinal} \end{align}

\begin{lem}
For all $x \in \T_N$, we have
 \begin{align} \mu_{N,K} ( \{ \eta \in \Omega_{N,K} \ : \ \eta (x) = 1 \}) &= \frac{K}{N}, \label{trivial1}\\  
 \mu_{N,K} (\{ \eta \in \Omega_{N,K} \ : \ \eta (x) = \eta (x+r) = 1 \}) &= \frac{K}{N} \cdot \frac{K-1}{N-(r-1)K-1} . \label{trivial2}
   \end{align}
\end{lem}
\begin{proof}
Making use of \eqref{wurst} and \eqref{Kardinal}, we get
 \[ \mu_{N,K} ( \{ \eta \in \Omega_{N,K} \ : \ \eta (x) = 1 \}) = \frac{ {N-(r-1)K-1 \choose K-1}}{(r-1){N-(r-1)K-1 \choose K-1} + {N- (r-1) K \choose K}} = \frac{K}{N} . \] 
Analogously to the above, combinatorics give \[ | \{ \eta \in \Omega_{N,K} \ : \ \eta (x) = \eta (x+r) = 1 \}| = {N-(r-1)K-2 \choose K-2} , \] and thus \begin{align}\label{KN} \begin{split} \mu_{N,K} (\{ \eta \in \Omega_{N,K} \ : \ \eta (x) = \eta (x+r) = 1 \}) &= \frac{{N-(r-1)K-2 \choose K-2}}{(r-1){N-(r-1)K-1 \choose K-1} + {N-(r-1)K \choose K}}  \\ &= \frac{K}{N} \cdot \frac{K-1}{N-(r-1)K-1} . \end{split} \end{align}
\end{proof}

The grand-canonical measure $\nu^{N}_{\rho}$ is obtained by an appropriate convex combination of the canonical measures $\mu_{N,K}$. We set 
\begin{align}\label{gcm}  \nu_{\rho}^{N} = \sum\limits_{K=0}^{\lfloor \frac{N}{r} \rfloor} \alpha_{N,K} (\rho) \mu_{N,K},  \end{align} 
where the $(\alpha_{N,K} (\rho))_{K \in \{0, \dots ,\lfloor \frac{N}{r} \rfloor \}}$ are the weights of a convex combination such that $\nu_{\rho}^{N} ( \{ \eta \in \Omega_N \ : \ \eta (x) = 1\} ) = \rho$.
 In the R1EP-case, the weights are $\alpha_{N_0,K} ( \rho)  = {N_0 \choose K} \rho^{K} (1-\rho)^{N_0 - K} $, which leads to the Bernoulli-product measure. 
In the RrEP-case, the sum only goes up to $\lfloor \frac{N}{r} \rfloor$. In order to still get weights which are concentrated around $\rho \cdot N$, the generic choice is 
\begin{align}\label{weights} \alpha_{N,K} ( \rho) := {\lfloor \frac{N}{r} \rfloor \choose K} \left( \frac{N}{\lfloor \frac{N}{r} \rfloor} \rho \right)^{K} \left( 1-\frac{N}{\lfloor \frac{N}{r} \rfloor} \rho \right)^{\lfloor \frac{N}{r} \rfloor - K} . \end{align}

Obviously we get \[ \sum\limits_{K=0}^{\lfloor \frac{N}{r} \rfloor} \alpha_{N,K}(\rho) = \left(\frac{N}{\lfloor \frac{N}{r} \rfloor} \rho + \left(1-\frac{N}{\lfloor \frac{N}{r} \rfloor} \rho\right)\right)^{\lfloor \frac{N}{r} \rfloor} = 1 , \] but also as desired for all 
$x \in \T_N$
\begin{align*} \nu_{\rho}^{N} ( \{ &\eta \in \Omega_N \  : \  \eta (x) = 1 \}) = \sum\limits_{K=0}^{\lfloor \frac{N}{r} \rfloor} \alpha_{N,K}(\rho) \cdot \mu_{N,K}( \{ \eta \in \Omega_N \  : \  \eta (x) = 1 \}) = \rho , \end{align*} by straightforward computations.
Next, we state a technical lemma that is going to be used in the following proposition.
\begin{lem}\label{technical}
For all $x \in \mathbb{R}$ and $q\in\bbN$,
 \begin{align}\label{claim} \sum\limits_{t=0}^{q} \frac{(-1)^{q-t}}{t! (q-t)!} \cdot \prod\limits_{j=0, j \neq t}^{q} (x-j) = 1 . \end{align} 
 \end{lem}
\begin{proof}
The left hand side of equation \eqref{claim} is a polynomial of degree $q$ in $x \in \mathbb{R}$, say $P(x)$. It is enough to show $P(0)=P(1)=\dots=P(q)=1$. But for $x \in \{0,1,\dots,q\}$ all but one term are $0$, namely $x=t$. Thus \[
P(x) = \frac{(-1)^{q-x}}{x! (q-x)!} \underbrace{(x-0) \cdot (x-1) \cdot \dots \cdot (x-(x-1))}_{= x!} \cdot \underbrace{(x-(x+1)) \cdot \dots \cdot (x-q)}_{= (-1)^{x-q} (q-x)!}  = 1.\]
\end{proof}

We will now address the problem of finding the limit of $\nu^{N}_{\rho}$-measures of configurations having particles on two fixed sites with minimal distance $r$.

\begin{prop}\label{cruciallemma}
Consider the RrEP on the torus $\T_N$. Let $\nu_{\rho}^{N}$ be the grand-canonical measure according to \eqref{gcm} with weights given by \eqref{weights}. Then, for all $x \in \T_N$,
 \begin{align}\label{ulli}
\nu_{\rho}^{N} (\{ \eta \in \Omega_N \ : \ \eta(x)=\eta(x+r)=1 \}) \to \frac{\rho^{2}}{1-(r-1) \rho} \quad \text{for } N \to \infty.
 \end{align}
\end{prop}
\begin{proof}
At first, bringing the demanded probability into a convenient form, we get

 \begin{align*}
\begin{array}{lcl} 
& &\nu_{\rho}^{N} (\{ \eta \in \Omega_N \ : \ \eta(x) = \eta(x+r)=1 \})  \\
 &\overset{\eqref{KN}}{=} &\sum\limits_{K=2}^{\lfloor \frac{N}{r} \rfloor} \frac{K}{N} \cdot  \frac{K-1}{N-(r-1)K-1} {\lfloor \frac{N}{r} \rfloor \choose K} \left( \frac{N}{\lfloor \frac{N}{r} \rfloor } \rho \right)^{K} \left( 1-\frac{N}{\lfloor \frac{N}{r} \rfloor } \rho \right)^{\lfloor \frac{N}{r} \rfloor - K} 
\\ &= &\frac{\lfloor \frac{N}{r} \rfloor \cdot (\lfloor \frac{N}{r}\rfloor - 1 )}{N} \sum\limits_{K=2}^{\lfloor \frac{N}{r} \rfloor} {\lfloor \frac{N}{r}\rfloor - 2 \choose K-2}  \left( \frac{N}{\lfloor \frac{N}{r} \rfloor } \rho \right)^{K} \left( 1-\frac{N}{\lfloor \frac{N}{r} \rfloor} \rho \right)^{\lfloor \frac{N}{r} \rfloor - K} \frac{1}{N - (r-1)K-1} \\
   &= &\sum\limits_{K=2}^{\lfloor \frac{N}{r} \rfloor}\frac{\lfloor \frac{N}{r} \rfloor (\lfloor \frac{N}{r} \rfloor - 1)}{N (N-(r-1)K-1)}  {\lfloor \frac{N}{r}\rfloor - 2 \choose K-2} \sum\limits_{s=0}^{\lfloor \frac{N}{r} \rfloor -K} \rho^{K+s} \cdot \left( (\frac{N}{\lfloor \frac{N}{r} \rfloor})^{K+s} (-1)^{s} {\lfloor \frac{N}{r} \rfloor - K \choose s}   \right) \\  &{}\overset{m:=K+s}{=}  &\sum\limits_{m=2}^{\lfloor \frac{N}{r} \rfloor} \frac{((r-1)\rho)^{m}}{(r-1)^{2}} \cdot a_N (m) ,
\end{array}
\end{align*}
where \[ a_N (m) := \frac{1}{(r-1)^{m-2}} \left( \frac{N}{\lfloor \frac{N}{r} \rfloor} \right)^{m} \sum\limits_{K=2}^{m}  {\lfloor \frac{N}{r} \rfloor - K \choose m-K} {\lfloor \frac{N}{r} \rfloor -2 \choose K-2} \frac{(-1)^{m-K} \lfloor \frac{N}{r} \rfloor ( \lfloor \frac{N}{r} \rfloor -1)}{N (N-(r-1)K-1)} \]
Clearly, it holds \[\sum\limits_{m=2}^{\lfloor \frac{N}{r} \rfloor} \frac{((r-1)\rho)^{m}}{(r-1)^{2}}  = \sum\limits_{j=0}^{\lfloor \frac{N}{r} \rfloor -2} \frac{((r-1)\rho)^{j+2}}{(r-1)^{2}}  = \rho^{2} \sum\limits_{j=0}^{\lfloor \frac{N}{r} \rfloor-2} ((r-1)\rho)^{j} \to \frac{\rho^{2}}{1-(r-1) \rho} , \] for $N \to \infty$ as $|(r-1)\rho|<1$.
Additionally, for fixed $m$, we have $\frac{((r-1)\rho)^{m}}{(r-1)^{2}} \cdot a_N (m) \geq 0$. 
Reformulating $a_N (m)$ further gives 
\begin{align*}
\begin{array}{llcl}
&a_N (m) &= &\frac{1}{(r-1)^{m-2}} \left( \frac{N}{\lfloor \frac{N}{r} \rfloor} \right)^{m} \frac{\lfloor \frac{N}{r} \rfloor !}{N (\lfloor \frac{N}{r} \rfloor - m)!} \sum\limits_{K=2}^{m} \frac{(-1)^{m-K}}{(m-K)! (K-2)! (N-(r-1)K-1)} \\ \\ & &\underset{t:=K-2}{\overset{q:= m-2}{=}} &\frac{1}{(r-1)^{q}}   \left(\frac{N}{\lfloor \frac{N}{r} \rfloor} \right)^{q+1} \frac{(\lfloor \frac{N}{r} \rfloor -1)!}{(\lfloor \frac{N}{r}\rfloor - q - 2)!} \sum\limits_{t=0}^{q} \frac{(-1)^{q-t}}{(q-t)! t! (N-2r+1-(r-1)t)} \\ \\
& &= &
\underbrace{{\textstyle \left(\frac{N}{\lfloor \frac{N}{r} \rfloor} \right)^{q+1} \frac{(\lfloor \frac{N}{r} \rfloor -1)!}{(\lfloor \frac{N}{r}\rfloor - q - 2)! \prod\limits_{j=0}^{q} (N - 2r + 1 - (r-1)j)    }}}_{=:E} \\
& & &\cdot \underbrace{{\textstyle \sum\limits_{t=0}^{q} \frac{(-1)^{q-t}}{t! (q-t)!} \frac{1}{(r-1)^{q}}\prod\limits_{j=0, \ j \neq t}^{q}  ((N - 2r + 1) - (r-1)j)       }}_{=:F}.
\end{array}
\end{align*}

On the one hand, it holds that for fixed $m$ (i.e. for fixed $q$) $E \to 1$ for $N \to \infty$, since both $\left( \lfloor \frac{N}{r} \rfloor \right)^{q+1} \prod\limits_{j=0}^{q} (N-2r+1-(r-1)j)$ and $N^{q+1} \frac{(\lfloor \frac{N}{r} \rfloor -1)!}{(\lfloor \frac{N}{r} \rfloor -q -2)!}$  are polynomials in $N \cdot \lfloor \frac{N}{r} \rfloor$ of degree $q+1$ with leading coefficient $1$. On the other hand, there actually holds equality for $F$, i.e. for $0 \leq q \leq \lfloor \frac{N}{r} \rfloor -2 $ we have $F=1$, due to Lemma \ref{technical}, putting $x = \frac{N-2r+1}{r-1}$. 
With $\frac{\left(\lfloor \frac Nr\rfloor-1\right)!}{\left(\lfloor \frac Nr\rfloor-q-2\right)!}=\prod_{j=0}^q\left(\lfloor \frac Nr\rfloor-1-j\right)$, we can rewrite and bound $E$
\begin{align*}
	E&= \prod_{j=0}^q\frac{N}{\lfloor \frac Nr\rfloor}\frac{\lfloor \frac Nr\rfloor-1-j}{N-2r+1-(r-1)j}
	\leq \left(\frac{N}{\lfloor \frac Nr\rfloor}\frac{\lfloor \frac Nr\rfloor-1}{N-2r+1}\right)^{m-1},
\end{align*}
where we used that $\frac{N}{\lfloor \frac Nr\rfloor}\frac{\lfloor \frac Nr\rfloor-1-j}{N-2r+1-(r-1)j}\leq\frac{N}{\lfloor \frac Nr\rfloor}\frac{\lfloor \frac Nr\rfloor-1}{N-2r+1}$ for all $j=0,\dots,q$.
Since $\frac{N}{\lfloor \frac Nr\rfloor}\frac{\lfloor \frac Nr\rfloor-1}{N-2r+1}\to 1$, there exists some $0<\lambda<1$ such that
\begin{align*}
	f_N(m):= \frac{((r-1)\rho)^{m}}{(r-1)^{2}}  \cdot a_N (m) \cdot \id_{\{m \leq \lfloor \frac{N}{r} \rfloor\}}<\frac{\lambda^m}{(r-1)^2},
\end{align*}
for all $N$ large enough.
Therefore, we can apply the dominated convergence theorem with $f(m):=\frac{((r-1)\rho)^{m}}{(r-1)^{2}} $ and $f_N (m) $, giving \[
\lim\limits_{N \to \infty} \int\limits_{\mathbb{N}\backslash \{0,1\}} f_N (m) \diff  \mathcal{C}  = \int\limits_{\mathbb{N}\backslash \{0,1\}} f(m) \diff  \mathcal{C} = \sum\limits_{m=2}^{\infty} \frac{((r-1)\rho)^{m}}{(r-1)^{2}}  = \frac{\rho^{2}}{1-(r-1)\rho} ,\] where $\mathcal{C}$ is the counting measure on $\mathbb{N}\backslash \{0,1\}$.
\end{proof}

Apart from knowing a stationary distribution on $\Omega_N$, we are also interested in defining a stationary distribution $\nu_{\rho}$ on the whole lattice $\mathbb{Z}$, i.e. on the configuration space 
\begin{align}\label{zed} \Omega_{\mathbb{Z}} := \left\{ \eta \in \{0,1\}^{\mathbb{Z}} \ : \ \eta (y) = 1 \ \Rightarrow \ \eta(y \pm k)=0 \ \forall k \in \{ 1,\dots,r-1\}, \forall y \in \mathbb{Z}  \right\} .\end{align}

We will do so by fixing a converging subsequence of $\nu^{N}_{\rho}$. As $\mathbb{Z}$ is countable and $\{0,1\}^{\mathbb{Z}}$ is equipped with the product topology, the latter is compact. Thus the above configuration space, as a closed subset of $\{0,1\}^{\mathbb{Z}}$, is compact as well. 
This means that the sequence $\nu^{N}_{\rho}$ of measures on $\mathcal{B}(\Omega_{\mathbb{Z}})$ is tight, implying by Prohorov's Theorem
that a converging subsequence $\nu^{N_k}_{\rho}$ indeed exists. 
It follows from propositions
I.1.8 and I.2.14 in \cite{Liggett85} that the stationarity property is preserved for $\nu_{\rho}$.

\section{Proof of Theorem \ref{thm:main}}\label{sec:proof}

The proof follows closely the standard method laid out in \cite{Kipnis99}. However, as usual, at each step it has to be checked that 
the relevant estimates hold for the process at hand. This is what we do in this section. 

We start with a calculation that will be fundamental throughout the proof. Fix a smooth $G: \T \to \bbR$, and consider the function 
$\eta \mapsto \langle \pi_\eta^N, G \rangle$ defined in \eqref{summeempirisches1}. We compute the result of applying $N^2 L_N$,
the generator of the process diffusively rescaled process $(\eta_{N^2 t})$, to this function: 

\begin{equation} \label{weiter}
\begin{split}
 & N^2 [L_N \langle \pi_{.}^{N} , G \rangle](\eta) = 
  N  \sum\limits_{x \in \T_N } G( \tfrac{x}{N}) \Big( - \frac12 \eta (x) (1- \eta (x-r)) - \frac12 \eta(x) (1-\eta (x+r))  \\ 
   & \qquad + \frac12 \eta (x-1) (1- \eta (x+r-1))  + \frac12 \eta (x+1) (1 - \eta (x-r+1))  \Big)  \\
 & = \frac{1}{2} \langle \pi_{\eta}^{N} , \Delta_N G \rangle + \frac{N}{2} \sum\limits_{x \in \T_N } \eta (x) \eta (x+r) \cdot \Big( G( \tfrac{x+r}{N}) + G (\tfrac{x}{N}) - G(\tfrac{x+1}{N}) - G(\tfrac{x+r-1}{N}) \Big) ,
\end{split} 
\end{equation}
where 
\begin{align*}
{\Delta}_N G (\tfrac{x}{N}) := N^{2} \left( G(\tfrac{x+1}{N}) + G(\tfrac{x-1}{N}) - 2 G(\tfrac{x}{N}) \right),
\end{align*}
is the discrete Laplacian. Note that the final term in \eqref{weiter} is zero for the SSEP. Summation by 
parts\footnote{the fact that this summation by parts can be carried out means that the model is of gradient type, a property that the higher dimensional models lack.}
now gives
\begin{align}\label{genempi2} \begin{split}
N^{2} [L_N \langle \pi_{.}^{N} , G \rangle](\eta) & = \frac{1}{2} \langle  \pi_{\eta}^{N} ,\Delta_N G \rangle + \frac{1}{2N} \sum\limits_{x \in \T_N } \eta (x) \eta (x+r) \cdot \left( \sum\limits_{j = x+1}^{x+r-1} \Delta_N G (\tfrac{j}{N}) \right) \\ 
& = \frac{1}{2N} \sum\limits_{x \in \T_N} \Delta_N G (\tfrac{x}{N}) \Big( \eta(x) + \eta(x-r+1) \eta(x+1) + \dots \\ 
& \qquad \qquad \qquad \qquad \qquad \dots + \eta (x-1) \eta (x-1+r) \Big)  .
\end{split} \end{align} 
We introduce the function 
\begin{align}\label{h} 
h(\eta) := \eta(0) + \eta(1-r) \eta(1) + \dots + \eta(-1) \eta (r-1), 
\end{align} 
and the shift $\tau_x: \Omega_{N,A} \to \Omega_{N,A}$, with $[\tau_x \eta](y) = \eta(x + y)$. With the convention that 
$\tau_x f (\eta) := f(\tau_x \eta)$ for functions $f$ on $\Omega_{N,A}$, we then have the concise formula 
\begin{equation} 
\label{eq: generator formula}
N^{2} [L_N \langle \pi_{.}^{N} , G \rangle](\eta) = \frac{1}{2N} \sum\limits_{x \in \T_N} \Delta_N G (\tfrac{x}{N}) \tau_x h (\eta).
\end{equation} 

This information, together with our calculations for the stationary measure, 
is already sufficient for guessing the form of the hydrodynamic equation. To see this, consider the fundamental martingale 
\begin{align} \label{martingale} \begin{split} M_{t} 
&= \langle \pi_{N^{2}t}^{N} , G \rangle - \langle \pi_0^{N} , G \rangle - \int\limits_0^{N^{2}t}  L_N \langle \pi_{s}^{N} , G \rangle \diff s \\ &
= \langle \pi_{N^{2}t}^{N} , G \rangle - \langle \pi_0^{N} , G \rangle -  \int\limits_0^{t} N^{2}  L_N \langle \pi_{N^{2} s}^{N} , G \rangle   \diff s . \end{split} \end{align}
of the real-valued process $(\langle \pi_{N^{2}t}^{N} , G \rangle)_{t \geq 0}$. We have just seen that 
\begin{equation} \label{eq:M_t explicit}
 M_{t} = \langle \pi_{N^{2}t}^{N} , G \rangle - \langle \pi_0^{N} , G \rangle -  \int\limits_0^{t} \frac{1}{2N} \sum\limits_{x \in \T_N} \Delta_N G (\tfrac{x}{N}) \tau_x h (\eta_{N^2 s})  \diff s. 
\end{equation}
We now pretend that we already know  the empirical measures to converge to a non-random measure $\rho(t,x) \dd x$, and  our aim is to determine 
$\rho$. Since the limit is not random, we can as well take expectations on both sides, and so the left hand side vanishes as it is a martingale 
starting in $0$. Rearranging and taking the limit $N \to \infty$ gives us 
\begin{equation} \label{eq:integrated}
\int_\T G(x) \rho(t,x) \, \dd x - \int_\T G(x) \rho(0,x) \, \dd x = \lim_{N \to \infty}
\int\limits_0^{t} \frac{1}{2N} \sum\limits_{x \in \T_N} \Delta_N G (\tfrac{x}{N}) \bbE(\tau_x h (\eta_{N^2 s}))  \diff s
\end{equation}
The sum over $x$ converges to a Riemann integral, the discrete Laplacian to the continuous Laplacian, and it remains to compute 
\begin{equation} \label{eq:expectation}
\begin{split} 
& \bbE(\tau_x h (\eta_{N^2 s})) = \E(\eta_{N^2 s}(x)) + \bbE(\eta_{N^2 s}(x-r+1) \eta_{N^2 s}(x+1)) + \ldots \\ 
& \qquad  + \bbE(\eta_{N^2 s}(x-1) \eta_{N^2 s}(x-1+r)).
\end{split}
\end{equation}
At this point, we make the heuristic {\em local equilibrium Ansatz}: we assume that if we know that the empirical density converges to 
$\rho > 0$ for some $x \in \bbT$, then in a small interval around $x$, the process $\eta_{N^2 s}(x)$ has, already for finite, large $N$, 
a distribution that is very close to the grand-canonical distribution $\nu_\rho^N$ that we derived in the previous section. This means 
$\bbE(\eta_{N^2 s}(x)) \to \rho(s,x)$ as $N \to \infty$ on the right hand side of \eqref{eq:expectation}, and that all the $(r-1)$ remaining 
terms on the right hand side are approximately equal to 
\[
\bbE_{\nu_{\rho(s,x)}^N}(\eta(x-r+1) \eta(x+1)) = \bbP_{\nu_{\rho(s,x)}^N}(\eta(x-r)=1, \eta(x)=1) \to \frac{\rho(s,x)^2}{1 - (r-1) \rho(s,x)}
\]
as $N \to \infty$, due to Proposition \ref{cruciallemma}. In total, for $r> 1$ the right hand side thus equals 
\[
\rho(s,x) + \frac{(r-1) \rho(s,x)^2}{1 - (r-1) \rho(s,x)} =  \frac{\rho(s,x)}{1 - (r-1) \rho(s,x)} = \frac{1}{r-1} 
\Big( \frac{1}{1 - (r-1) \rho(s,x)} - 1 \Big).
\]
We use this in \eqref{eq:integrated}, take the limit $N \to \infty$, and perform two further integrations by parts to get 
\[
\int_\T G(x) \rho(t,x) \, \dd x - \int_\T G(x) \rho(0,x) \, \dd x = 
\frac12\int\limits_0^{t} \dd s \int_\T  \dd x G (x) [\Delta \Phi(\rho)](x,s),
\]
with $\Phi(y) = \frac1{(r-1)}\frac{1}{1 - (r-1) y}$. Thus $\rho$ fulfills the PDE \eqref{PDE}.

We now have to close the gaps left by this heuristic derivation. We start with an estimate that will be useful several times below. 
Consider a smooth function 
$G:[0,T] \times \T \to \bbR$, and define the process
\begin{align} \label{uber} 
M_{t}^{G,N} 
 =  \langle \pi_{N^{2} t}^N , G ( t, \cdot) \rangle - \langle \pi_0^N , G(0,\cdot) \rangle - \int\limits_{0}^{t} (\partial_s + N^{2} L_{N} ) \langle \pi^N_{N^{2} s} , G (s,\cdot) \rangle \diff s,
\end{align}
which is a slight extension of the process defined in \eqref{martingale}, and in fact is again a martingale, namely the 
fundamental martingale of the process 
$(\langle \pi_{N^{2}t}^{N} , G \rangle)_{t \geq 0}$, only this time with a time-dependent $G$. Our first result is that this martingale
becomes constant as $N \to \infty$ in the following strong sense.

\begin{prop} \label{prop:martingales}
Let $\tau$ be a stopping time with respect to the natural filtration of $M_t^{G,N}$, and $\theta > 0$. Then 
\[
\lim_{N \to \infty} \bbE\Big( \big(M_{\tau + \theta}^{G,N} - M_{\tau}^{G,N}\big)^2 \Big)  = 0. 
\]
\end{prop} 

\begin{proof}
Because $(M_t^{G,N})$ is a martingale, we have that 
\[
\bbE\Big( \big(M_{\tau + \theta}^{G,N} - M_{\tau}^{G,N}\big)^2 \Big) = \bbE\Big( \big(M_{\tau + \theta}^{G,N}\big)^2 \Big) -
\bbE \Big( \big(M_{\tau}^{G,N}\big)^2 \Big).
\]
In Lemma 5.1 from Appendix 1 of \cite{Kipnis99} it is proved (in a more general situation) that the process  
\begin{align}\label{Ntg} 
Q_{t}^{G,N} = 
(M_{t}^{G,N})^{2} - N^2 \int\limits_{0}^{t} L_N \langle \pi_{N^{2} s}^{N} , G(s,.) \rangle^{2} - 2  \langle \pi^{N}_{N^{2} s} , G(s,.) \rangle L_N \langle \pi^{N}_{N^{2} s} , G(s,.) \rangle \diff s
\end{align} 
also is a martingale. Therefore,
\[
\bbE\Big( \big(M_{\tau + \theta}^{G,N} - M_{\tau}^{G,N}\big)^2 \Big) = \bbE \Big( 
\int\limits_{\tau}^{\tau + \theta}  B_{N^2 s}^G \diff s \Big) \leq \int_0^T \bbE(|B_{N^2 s}^G|)  \dd s,
\]
with 
\[
B^G_{N^2 s} = N^2 \big(  L_N \langle \pi_{N^{2} s}^{N} , G(s,.) \rangle^{2} - 2  \langle \pi^{N}_{N^{2} s} , G(s,.) \rangle L_N \langle \pi^{N}_{N^{2} s} , G(s,.) \rangle \big).
\]
In order to compute $B_{N^{2}s}^{G}$, let us first notice that the $s$-dependence of $G$ plays no role for applying the generator, 
so we will drop it from the notation. For sites $x,y \in \T_N, \ |x-y|=1,$ with $\eta_{N^{2}s}(x) = 1$ and such that a jump to $y \in \T_N$ is allowed under the dynamics, we have 
\begin{align}
\label{hilfsding}
\begin{split}
 &\left( \sum\limits_{z \in \T_N} \eta_{N^{2} s}^{x,y} (z)  G(\tfrac{z}{N}) \right)^{2} - \left( \sum\limits_{z \in \T_N} \eta_{N^{2} s} (z) G(\tfrac{z}{N}) \right)^{2}  \\ = &\left( \sum\limits_{z \in \T_N} (\eta_{N^{2}s}^{x,y} (z) + \eta_{N^{2}s}(z)) G(\tfrac{z}{N}) \right) \cdot \left( \sum\limits_{z \in \T_N} (\eta^{x,y}_{N^{2}s} (z) - \eta_{N^{2}s} (z)) G(\tfrac{z}{N}) \right)  \\ = & \left(\left(2 \sum\limits_{z \in \T_N} \eta_{N^{2}s} (z) G(\tfrac{z}{N})\right) +  G(\tfrac{y}{N}) -G(\tfrac{x}{N})\right) \cdot \left( G(\tfrac{y}{N}) - G(\tfrac{x}{N}) \right)  .
\end{split}
\end{align}
We can use this equation to derive
\begin{align} 
\label{affect}
\begin{split}
&N^{2} L_N \left( \langle \pi_{N^{2} s}^{N},G\rangle^{2} \right) \\ = &\frac12 \sum\limits_{\substack{x,y \in \T_N, \\ |x-y|=1}} \eta_{N^{2} s} (x) \left( \prod\limits_{k=0}^{r-1} (1-\eta_{N^{2} s}(y+(y-x)k)) \right) \cdot \\ & \qquad \qquad \cdot \left( \left( \sum\limits_{z \in \T_N} \eta_{N^{2} s}^{x,y} (z)  G(\tfrac{z}{N}) \right)^{2} - \left( \sum\limits_{z \in \T_N} \eta_{N^{2} s} (z) G(\tfrac{z}{N}) \right)^{2} \right) \\ 
= &\frac12 \sum\limits_{\substack{x,y \in \T_N, \\ |x-y|=1}} \eta_{N^{2} s} (x) \left( \prod\limits_{k=0}^{r-1} (1-\eta_{N^{2} s}(y+(y-x)k)) \right) \left( G(\tfrac{y}{N}) - G(\tfrac{x}{N})\right)^{2} \\ & + 2 N^2 \langle \pi_{N^{2} s}^{N} , G \rangle L_N \langle \pi_{N^{2} s}^{N} , G \rangle ,
\end{split}
\end{align}
from which we can conclude 
\begin{align}
\label{B_konkret}
B_{N^{2}s}^{G} = \frac12 \sum\limits_{\substack{x,y \in \T_N, \\ |x-y|=1}} \eta_{N^{2} s} (x) \left( \prod\limits_{k=0}^{r-1} (1-\eta_{N^{2} s}(y+(y-x)k)) \right) \left( G(\tfrac{y}{N}) - G(\tfrac{x}{N})\right)^{2} .
\end{align}
Thus, because $G$ is smooth, and because $\eta_s(x) \in \{0,1\}$, we have 
\[
\bbE (|B_{N^{2}s}^{G}|)  \leq \frac12 \sum\limits_{\substack{x,y \in \T_N, \\ |x-y|=1}} \left( G(s,\tfrac{y}{N}) - G(s,\tfrac{x}{N})\right)^{2} 
= O(N^{-1}).
\]
This concludes the proof.
\end{proof}

The first proper step of proving hydrodynamic limits is to show relative compactness. For conceptual clarity, 
we introduce the maps 
\[
Y_N: D([0,T], \Omega_{N,r}) \to D([0,T], \caM_+(\T)), \qquad (\eta_t^N)_{t \leq T} \mapsto (\pi_{\eta_t}^N)_{t \leq T}.
\]
On $D([0,T], \Omega_{N,r})$, we consider the distribution of the (diffusively rescaled) 
Markov process with generator \eqref{gen} and initial distribution 
$(\eta_0^N)$ so that the sequence $(\pi^N_{\eta_0})_{N \in \bbN}$ is associated to some profile $\rho_0$. Thus, each $Y_N$
becomes a random variable. Its distribution is thus a probability measure on $D([0,T], \caM_+(\T))$. Recall that we equip 
$\caM_+(\T)$ with the (metrization of the) weak topology, and $D([0,T], \caM_+(\T))$ with the Skorohod topology, so we know what 
continuous functions on $D([0,T], \caM_+(\T))$ are. In particular, we have a notion of weak convergence of measures on $D([0,T], \caM_+(\T))$. 
We now have 

\begin{prop} \label{prop: relative compactness}
The sequence of distributions of $(Y_N)_{N \in \bbN}$ is relatively compact in the weak topology. 
\end{prop}

\begin{proof}
We can test weak convergence with a restricted set of continuous test functions. Namely, a sequence $(\pi^N_t)$ 
of probability measures on $D([0,T], \caM_+(\T))$ is relatively compact if and only if, for all smooth $G: \T \to \bbR$, the sequence of 
image measures under the map $(\pi_t^N)_{t \leq T} \mapsto (\langle \pi_t^N, G \rangle)_{t \leq T}$ is relatively compact as probability measures on 
$D([0,T], \bbR)$. The proof uses the fact that the topology on $\caM_+(\T)$ can be metrized through formula \eqref{eq:metric}, and can be found 
as Proposition 1.7 in Chapter 4 of \cite{Kipnis99}. We thus fix a smooth $G:\bbT \to \bbR$ and consider the sequence of real-valued stochastic processes
$(\langle \pi_{\eta_{N^2 t}}^N, G \rangle)_{t \leq T}$. 

For such processes, we have Prohorov's criterion, which states that the sequence of distributions of 
real-valued processes $(Z_t^N)_{t \leq T, N \in \bbN}$ is tight if we can check two conditions. The first is simply that 
the sequence of marginal distributions under 
$(Z_t^N)_{t \in [0,T]} \mapsto Z_s$ is tight, which in our case is immediate since all the distributions of 
$\eta \mapsto \langle \pi_{\eta_{N^2 s}}^N, G \rangle$ have joint compact support in $[- \| G\|_\infty, \| G \|_\infty]$. 
The second criterion requires to control the modulus of continuity of the processes in the limit $N \to \infty$, and, following 
\cite{Kipnis99}, we use a criterion of Aldous \cite{Aldous78}: it is enough to check that 
\begin{align}
\lim\limits_{\gamma \to 0} \limsup\limits_{N \to \infty} \sup\limits_{\tau \in \mathcal{T}_{T},  \theta \leq \gamma} \bbP \left( | \langle \pi^{N}_{N^{2}(\tau + \theta)} , G \rangle - \langle \pi^{N}_{N^{2} \tau} , G \rangle | > \varepsilon \right) = 0 \qquad \forall \varepsilon > 0  , \label{23} 
\end{align}
where $\mathcal{T}_T$ is the set of all stopping times (with respect to the natural filtration of $(\langle \pi_t^N, G \rangle)_{t \leq T}$) bounded by $T$. 

Using equation \eqref{martingale}, we obtain 
\begin{align}\label{split}\begin{split}
&\bbP \left( | \langle \pi^{N}_{N^{2}(\tau + \theta)} , G \rangle - \langle \pi^{N}_{N^{2} \tau} , G \rangle | > \varepsilon \right) \\ 
\leq & \bbP \left( | M^{G,N}_{N^{2} (\tau + \theta)} - M^{G,N}_{N^{2} \tau} | > \frac{\varepsilon}{2} \right) + 
\bbP \left( \left| \int\limits_{\tau}^{\tau + \theta} N^{2} L_N \langle \pi^{N}_{N^{2}s} , G \rangle \diff s \right| > \frac{\varepsilon}{2} \right)
\end{split}\end{align}
Now we can use \eqref{eq: generator formula} and the fact that $|h(\eta)| \leq 1$ for $\eta \in \Omega_{N,r}$, and obtain
\[
\begin{split} 
 \Big| \int\limits_{\tau}^{\tau + \theta} N^{2} L_N \langle \pi^{N}_{N^{2}s} , G \rangle \diff s \Big| & \leq \sup_{t \in [0,T-\theta]}
 \int_t^{t + \theta} \Big| \frac{1}{2N} \sum\limits_{x \in \T_N} \Delta_N G (\tfrac{x}{N})  \tau_x h (\eta_{N^{2}s})  \Big|\, \dd s \\ 
 & \leq \theta  \frac{1}{2N} \sum\limits_{x \in \T_N}\Big| \Delta_N G (\tfrac{x}{N})\Big|.
 \end{split}
\]
The final sum converges to $\frac12 \int_\T \left|\partial^2_x G(x)\right| \, \dd x$, and since $\theta \leq \gamma$, the second term on the right hand side of 
\eqref{split} vanishes in the limit. 

For the first term, we can use Proposition \ref{prop:martingales} (for the case of $G$ constant in time) to obtain 
\[
\lim_{N \to \infty} \bbE \Big( \big| M^{G,N}_{N^{2} (\tau + \theta)} - M^{G,N}_{N^{2} \tau} \big|^2 \Big) = 0.
\] 
Then an application of Chebyshev's inequality finishes the proof.
\end{proof}

The next step is to show that for each convergent subsequence $(Y_{N_k})_{k \in \bbN}$, the limit has to be of the form 
$(\rho(t,x) \dd x)_{t \leq T}$, where $\rho$ is a weak solution of the PDE \eqref{PDE}. We suppress the 
subsequence notation and assume that $(Y_N)$ itself converges. Let us write $(\pi_t)_{t \in [0,T]}$ for the limiting element 
of $D([0,T], \caM_+)$. We have 

\begin{lem} \label{lem:regular}
For all $t \in [0,T]$, $\pi_t$ is almost surely absolutely continuous with respect to Lebesgue measure on $\T$. 
\end{lem}
\begin{proof}
Let $G:\T \to \bbR$ be smooth. The map 
\[
D([0,T], \caM_+) \to \bbR, (\nu_t)_{t \in [0,T]} \mapsto \sup\limits_{t \in [0,T]} | \langle \nu_t , G \rangle | 
\]
is continuous with respect to the Skorohod-topology, and thus the set  
\[
B = \Big \{ (\nu_t)_{t \in [0,T]} \in D([0,T],\mathcal{M}_{+}) \ : \ \sup\limits_{t \in [0,T]} | \langle \pi_t , G \rangle | \leq \int\limits_{\T} | G(u) | \diff u \Big\}
\] 
is closed.

As there is at most one particle per site, $(\pi^N_{N^2 t})_{t \leq T} \in B$ almost surely, for all $N$. 
By the Portmanteau-Theorem for closed sets,
\begin{align*}
1 = \limsup\limits_{N \to \infty} \bbP \Big( \sup\limits_{t \in [0,T]} | \langle \pi_{N^2 t}^N , G \rangle | \leq \int\limits_{\T} | G(u) | \diff u 
\Big) 
 \leq \bbP \Big( \sup\limits_{t \in [0,T]} | \langle \pi_t , G \rangle | \leq \int\limits_{\T} | G(u) | \diff u \Big).
\end{align*}
\end{proof}
The previous result allows us to write $\pi_t(\dd x) = \rho(t,x) \dd x$ for all $t \in [0,T]$. Next we show that, almost surely, 
$\rho(t,x)$ solves the PDE \eqref{PDE}. First of all, the assumption that $(\pi_{\eta_0^N})$ is associated to the profile 
$\rho_0$ guarantees that $\rho(0,x) = \rho_0(x)$. Then, we recall the martingale $M^{G,N}$ from \eqref{uber} and note that by first applying 
Chebyshev's inequality, then using Doob's $L^2$-inequality, and then applying Proposition \ref{prop:martingales}, 
we obtain 
\[
\lim_{N \to \infty} \bbP(\sup_{t \in [0,T]} |M^{G,N}_{N^2 t}| > \eps ) = 0
\]
for all $\eps > 0$. Inserting the expression \eqref{eq:M_t explicit}, and adjusting for the fact that $G$ depends on time, we thus find that 
\begin{align} \label{toclose}
\begin{split}
\lim\limits_{n \to \infty} \bbP  \Big( & \sup\limits_{t \in [0,T]} \Big| \langle \pi^{{N_{k}}}_{N_{k}^{2}t} , G \rangle - \langle \pi^{{N_{k}}}_0 , G \rangle - \int\limits_{0}^{t} \langle \pi^N_{N^2 s}, \partial_s G \rangle \diff s  \\ &- \frac12  \int\limits_{0}^{t} \frac{1}{N} \sum\limits_{x \in \T_{N}}  \Delta_N G (s,\tfrac{x}{N}) \tau_x h (\eta^N_{N^{2}s} ) \diff s   \Big| > \delta \Big) = 0. 
\end{split}
\end{align}

At this point we need the replacement lemma due to Guo, Papanicolaou and Varadhan (Lemma 3.3 in \cite{Guo88}). 
We introduce the local empirical density of range $2l + 1$ centred at x
\[
\eta^{l} (x) = \frac{1}{2l + 1} \sum\limits_{|y-x| \leq l} \eta (y)  .
\]
The replacement lemma requires the assumption that for each $N$, 
there exists a stationary distribution $\nu^N$ of $(\eta^N_t)$, such that when the distribution of the sequence $(\eta_0^N)$ of initial configurations is denoted by $(\mu^N)$, then the relative entropy inequality 
\[
H(\mu^N | \nu^N) \leq C_0 N
\]
holds for some $C_0 > 0$ and all $N$. This inequality is actually fulfilled for all sequences $(\mu^N)$ in our model: using the formula 
\[
\begin{split}
H(\mu^N | \nu^N) & = \sum_{\eta \in \Omega_{N,r}} \log\Big(\frac{\mu^N(\{\eta\})}{\nu^N(\{\eta\})} \Big) \mu^N(\{\eta\}) \leq  \\
& \leq \sum_{\eta \in \Omega_{N,r}} | \log \nu^N(\{\eta\}) | \mu^N(\{\eta\}) \leq \max_{\eta \in \Omega_{N,r}} |\log \nu^N(\{\eta\})|, 
\end{split}
\]
all we need to check is that $\nu^N(\eta) \geq \e{- C_0 N}$ uniformly in $\eta \in \Omega^{N,r}$ and $N \in \bbN$. It is easy to check that 
$\nu_\rho^N$ from \eqref{gcm} has this property for all $\rho > 0$.   

The statement of the replacement lemma then is that for every $\delta > 0$, and for all sequences $\eta_0^N$ of 
initial configurations,
\begin{align*}
\limsup\limits_{\varepsilon \to 0} \limsup\limits_{N \to \infty} \PP_{\mu^{N}} \Big( \int\limits_{0}^{t} \frac{1}{N} \sum\limits_{x \in \T_N} \tau_x V_{\varepsilon N} (\eta_s) \ ds \geq \delta \Big) = 0  ,
\end{align*}
where
\[
V_l (\eta) = \Big| \Big( \frac{1}{2l + 1} \sum\limits_{|y| \leq l} \tau_y h( \eta ) \Big) - \Phi(\eta^{l} (0)) \Big|, \quad 
\text{with} \quad \Phi(\rho) = \frac{\rho}{1 - (r-1) \rho}.
\]

We use this in the following way: first, we replace the discrete Laplacian in equation \eqref{toclose} with the actual Laplacian. Since $G$ is 
smooth, this only introduces an error of order $1/N$. With this, we obtain 
\begin{align*}
& \frac{1}{2 N} \sum\limits_{x \in \T_{N}}  \Delta_{N} G (s,\tfrac{x}{N}) \tau_x h (\eta_{N^{2}s} ) \\
& \quad = \frac{1}{2N} \sum\limits_{x \in \T_{N}} \tfrac{1}{2 \varepsilon N + 1} \sum\limits_{|y-x| \leq \varepsilon N} \partial_{x}^{2}  G (s,\tfrac{y}{N}) \tau_y h(\eta_{N^{2} s}) \ + \ \mathcal{O} (\frac{1}{N}) \\
& \quad = \frac{1}{2 N} \sum\limits_{x \in \T_{N}} \partial_{x}^{2}  G(s,\tfrac{x}{N}) \tfrac{1}{2 \varepsilon N + 1} \sum\limits_{|y-x| \leq \varepsilon N} \tau_y h(\eta_{N^{2} s} ) \ + \mathcal{O} (\varepsilon) + \mathcal{O} (\tfrac{1}{N}) 
\end{align*}
Defining 
\[ \iota_{\varepsilon} (u) := \frac{1}{2 \varepsilon} \id_{\{|u| \leq \varepsilon\}}  ,
\]
one gets approximately (for large N) \begin{align*}
\eta_{N^{2} s}^{\varepsilon N} (0) &= \frac{1}{2 \varepsilon N + 1} \sum\limits_{|y| \leq \varepsilon N} \eta_{N^{2} s} (y) \\ &\sim \frac{1}{2 \varepsilon N} \sum\limits_{|y| \leq \varepsilon N} \eta_{N^{2} s} (y) = \frac{1}{N} \sum\limits_{x \in \T_N} \eta_{N^{2} s} (x) \cdot \id_{\{|\frac{x}{N}| \leq \varepsilon\}} \cdot \frac{1}{2 \varepsilon} =  \langle \pi_{N^{2} s}^{N} , \iota_{\varepsilon} \rangle  .
\end{align*}

Hence, applying the replacement lemma, replacing the local empirical density by a function of the empirical measure and lightening the notation by putting $a=\langle \pi_{N^{2} s}^{N} , \iota_{\varepsilon} \rangle $ for the parameter of the grand-canonical measure, the integral in  \eqref{toclose} is equal to
\begin{align*}
\frac{1}{2N_{k}} \sum\limits_{x \in \T_{N_{k}}} \partial_{x}^{2}  G (s,\frac{x}{N_k}) \tau_x \Phi (a) \ + \ R_{N, \varepsilon}  ,
\end{align*}
where $R_{N, \varepsilon}$ is an expression which vanishes in probability as $N_k \to \infty$ and $\varepsilon \to 0$ afterwards. 
Furthermore, replacing the sum by an integral, we only make an error of order $\mathcal{O}(\frac{1}{N})$.  In conclusion, 
\begin{align}
\label{honig100}
\begin{split}
\limsup\limits_{\varepsilon \to 0} \limsup\limits_{N \to \infty} \bbP  \Big( &  \Big| \langle \pi_{N^2 t}^N , G \rangle - \langle \pi_0 , G \rangle - \int\limits_{0}^{t} \langle \pi^N_{N^2 s}, \partial_s G \rangle \diff s   \\ &- \frac12  \int\limits_{0}^{t} \int\limits_{\T}  \partial_{x}^{2} G(s,u) \Phi \left(\int\limits_{\T} \iota_{\varepsilon} (v - u)  \pi^N_{N^2 s}(\diff v)  \right) \diff u \diff s   \Big| \geq \delta \Big) = 0. 
\end{split}
\end{align}
For each $\varepsilon > 0$ the map \begin{align*}
	(\pi_s)_{s \in [0,T]} \mapsto &\langle \pi_t , G \rangle - \langle \pi_0 , G \rangle - \int\limits_{0}^{t} \langle \pi_s, \partial_s G \rangle \diff s  \\ &- \frac12 \int\limits_{0}^{t}\int_{\bbT}  \partial_{x}^{2} G(s,u) \Phi\left(\int\limits_{\T} \iota_{\varepsilon} (v - u)  \pi_s(\diff v)  \right)\diff u \diff s 
\end{align*}
is continuous with respect to the Skorohod topology on $D([0,T], \caM_+)$. 
Thus the Portmanteau-Theorem is applicable, giving 
\begin{align}\label{honig}\begin{split}
\limsup\limits_{\varepsilon \to 0} \bbP  \Big( &  \Big| \langle \pi_t , G \rangle - \langle \pi_0 , G \rangle - \int\limits_{0}^{t} \langle \pi_s, \partial_s G \rangle \diff s   \\ &- \frac12  \int\limits_{0}^{t} \int\limits_{\T}  \partial_{x}^{2} G(s,u) \Phi (\int\limits_{\T} \iota_{\varepsilon} (v - u)  \pi_s(dv)  ) \diff u \diff s   \Big| \geq \delta \Big) = 0. \end{split}
\end{align}
The final goal in this step consists in getting rid of the $\varepsilon$, obtaining a function of the density.
Pointwise, by Lebesgue's differentiation theorem, we see that
\begin{align} \label{approxxi}
\langle \pi^{N}_{N^{2} s} , \iota_{\varepsilon} \rangle &= \frac{1}{2 \varepsilon} \int\limits_{\T} \id_{[- \varepsilon, \varepsilon]} (u) \ \pi^{N}_{N^{2} s} (du) \underset{N \to \infty}{\rightarrow} \frac{1}{2 \varepsilon} \int\limits_{\T} \id_{[- \varepsilon, \varepsilon]} (u) \cdot \rho (s,u) \diff u \underset{\varepsilon \to 0}{\rightarrow} \rho (s,0),
\end{align}
since we know that the limit $(\pi_t)$ is almost surely absolutely continuous. 

Since $\Phi$ is continuous, $\rho (s,\cdot ) \in [0,\frac{1}{r}]$ and $|\partial_{x}^{2} G | < \infty$, so by dominated convergence for $\varepsilon \downarrow 0$ we almost surely have 
\[ 
\int\limits_{0}^{t} \int\limits_{\T} \partial_{x}^{2}  G (s,u) \Phi ( \int\limits_{\T} \iota_{\varepsilon} (v-u) \pi_s (\diff v )) \diff u \diff s  \to  \int\limits_{0}^{t} \int\limits_{\T} \partial_{x}^{2}  G(s,u) \Phi ( \rho (s,u)) \diff u \diff s 
 \]
We conclude that 
\[ 
\limsup\limits_{\varepsilon \to 0} \bbP \Big( \Big| \int\limits_{0}^{t} \int\limits_{\T} \partial_{x}^{2} G (s,u) (\Phi (\int\limits_{\T} \iota_{\varepsilon} (v-u) \ \pi_s (dv)) - \Phi ( \rho (s,u)) )  \diff u \diff s \Big| > \delta \Big) = 0 \]
for all $\delta > 0$ and then let $\delta \to 0$. 

As a result, we have shown that all limit points $(\pi_t)$ of the empirical measures have a density $\rho(t,.)$ that almost surely 
satisfies the equation 
\[
\begin{split}
&  \int_\T \rho(t,x) G(t,x) \, \dd x -  \int_\T \rho(0,x) G(0,x) \, \dd x  = \\ 
& \quad =  \int_{0}^{t} \dd s 
 \int_\T \dd x \rho(s,x) \partial_s G(s,x) + \frac12  \int_\T \dd x \int_{0}^{t} \dd s \,  \Phi(\rho(s,x)) \partial_x^2 G (s,x).
\end{split}
\]
In other words, $\rho$ is almost surely a weak solution of \eqref{PDE}.
By Theorem 4.4 of Appendix A2.4 in \cite{Kipnis99} the uniqueness of the weak solution of \eqref{PDE} follows.
The proof is thus finished.

\section{Numerical approximation of the diffusion matrix in two dimensions} \label{sec:2dim}
We want to numerically approximate the diffusion matrix of the hydrodynamic limit of the finite range exclusion process in two dimensions with different rectangular exclusion zones $A$.
For this purpose, we assume that, for a given $A\subset\bbZ^2$, the hydrodynamic limit exists and is described by a nonlinear diffusion equation of the following form:
\begin{align}\label{eqn:diff_eq}
	\begin{cases}
		\begin{aligned}
			\partial_t\rho&=\operatorname{div}(D_A(\rho)\nabla\rho) &&\text{on }[0,T]\times\bbT^2,\\
			\rho(0,\cdot)&=\rho_0(\cdot) &&\text{on }\bbT^2,
		\end{aligned}
	\end{cases}
\end{align}
where $D_A:\bbR\to\bbR^{2\times2}$.
In \cite{fernandes2007}, a phase transition is observed for $A=\{0,1\}^2$ and $A=\{0,1,2\}^2$ at a critical density $\rho_c^A$.
For $\rho>\rho_c^A$, a columnary order emerges either along the $x$-axis or the $y$-axis.
In this partially ordered phase the mobility depends not only on the local density, but also on the direction of the ordering.
Since \eqref{eqn:diff_eq} describes the macroscopic evolution only in terms of the density, it does not account for the microscopic ordering and therefore cannot describe the hydrodynamic limit in the partially ordered phase.
Proposing a hydrodynamic limit equation in the partially ordered phase is beyond the scope of this work.
Hence, we restrict our analysis to the low density phase and assume that
\begin{align*}
	0\leq \rho_0<\rho_c^A\quad\text{a.e. on }\bbT^2.
\end{align*}
This implies that the solution $\rho$ to \eqref{eqn:diff_eq} is bounded by $\rho_c^A$ for a.e. $t\in[0,T]$ (see Lemma \ref{lemma:max_princ}).
Without loss of generality, we can assume that $D_A(\rho)$ is symmetric. 
To see this, let $M(\rho) := \frac12(D_A(\rho)-D_A(\rho)^T)$ and $F(r):=\int_0^rM(s)\diff s$.
Let $\rho\in H^1(\bbT^2)$ and $G:\bbT^2\to\bbR$ smooth.
Integration by parts yields
\begin{align*}
	\int_{\bbT^2}M(\rho)\nabla\rho\cdot\nabla G\diff x
	&=\sum_{i,j=1}^2\int_{\bbT^2} M_{i,j}(\rho)\partial_j\rho\partial_i G\diff x
	=-\sum_{i,j=1}^2\int_{\bbT^2} F_{i,j}(\rho)\partial_i\partial_j G\diff x\\
	&=\sum_{i,j=1}^2\int_{\bbT^2} F_{j,i}(\rho)\partial_j\partial_i G\diff x
	=-\sum_{i,j=1}^2\int_{\bbT^2} M_{j,i}(\rho)\partial_i\rho\partial_j G\diff x\\
	&=- \int_{\bbT^2}M(\rho)\nabla\rho\cdot\nabla G\diff x.
\end{align*}
Therefore by density, 
\begin{align*}
	\int_{\bbT^2}M(\rho)\nabla\rho\cdot\nabla G\diff x=0\quad\forall \rho,G\in H^1(\bbT^2),
\end{align*}
and the weak formulation is the same if we replace $D_A(\rho)$ by the symmetric part $\frac12(D_A(\rho)+D_A(\rho)^T)$.
We make the following assumptions for the diffusion matrix $D_A$.
\begin{assumption}\label{ass:diffusion_tensor}
	Assume that $D_A:[0,\rho_c^A)\to\bbR^{2\times2}$ is continuous.
	Moreover, $D_A$ is uniformly elliptic and bounded, i.e. there are $0<\lambda<\Lambda$ such that
	\begin{align*}
		\lambda|\xi|^2\leq\xi^TD_A(\rho)\xi\leq\Lambda |\xi|^2\quad\forall \xi\in\bbR^2,\,\rho\in[0,\rho_c^A).
	\end{align*}
\end{assumption}

Our main assumption on the hydrodynamic limit is 

\begin{assumption}\label{ass:hydrodynamic_limit}
	For every $\rho_0:\bbT^2\to[0,\rho_c^A)$ smooth, there exists a sequence $\eta_0^N$ of random elements in $\Omega_{N,A}$ that is associated to the profile $\rho_0$.
	Additionally, for every sequence $\eta_0^N$ of random elements in $\Omega_{N,A}$ that is associated to the profile $\rho_0$ the hydrodynamic limit of the 
	empirical densities of $(\eta_t)$ exists, is deterministic, and is given by $t \mapsto \rho(t,x) \dd x$, where $\rho$ is  
	a solution of \eqref{eqn:diff_eq}, where $D_A$ fulfills Assumption \ref{ass:diffusion_tensor}.
\end{assumption}

\begin{bem}
	\begin{enumerate}
		\item[(a)] Assumption \ref{ass:diffusion_tensor} collects regularity and non-degeneracy properties of the diffusion matrix.
			In particular, continuity of the diffusion matrix is the regularity needed for the linearization argument used below.
			The boundedness and ellipticity assumptions are used in the energy estimates in Appendix \ref{appendix:pde}.
		\item[(b)] Assumption \ref{ass:hydrodynamic_limit} states that the hydrodynamic limit exists, and that it is
		given by the solution to the PDE. The existence is the tricky part, but once we accept it, it is quite reasonable to assume 
		that the limiting dynamics is indeed a non-linear diffusion, and that the diffusion matrix has the desired regularity properties. 
		The former is expected to hold for all symmetric particle dynamics with a unique stationary distribution, and the latter should then 
		follow from the jump rates being uniformly bounded away from zero at low densities. Of course, to prove these things is likely
		very difficult. 		
	\end{enumerate}  
\end{bem}

Below, we will show that under these assumptions, for a rectangular exclusion zone $A$, the diffusion matrix is diagonal, 
while for a square exclusion zone it is a scalar multiple of the identity.\\
More precisely, we show that the particle process is invariant under reflection across an axis for rectangular $A$.
For a square exclusion zone, the process is in addition invariant under rotation by $\frac\pi2$.
We then show that these symmetries are inherited by the hydrodynamic limit and consequently impose corresponding algebraic constraints on the diffusion matrix.\\

\subsection{Symmetries of the hydrodynamic limit}
The next lemma proves that if the exclusion set $A$ is invariant under the action of an orthogonal matrix $U\in\bbZ^{d\times d}$, then the action of the Markov semigroup $P_t^N$ commutes with the action of $U$.
To prove this it is sufficient to prove that the generator \eqref{eqn:gen_A} commutes with the action of $U$.
\begin{lemma}\label{lemma:dist_rotated}
	Let $N\in\bbN$ and $\varphi_N:\bbT_N^d\to\bbT_N^d,x\mapsto Ux$ with $U\in\bbZ^{d\times d}$ orthogonal.
	If $UA=A+a$ for some $a\in\bbZ^d$, then for every probability measure $\mu^N$ on $\Omega_{N,A}$ and for every $t\geq 0$
	\begin{align*}
		(\mu^NP_t^N)\circ T_{\varphi_N}=(\mu^N\circ T_{\varphi_N})P_t^N,
	\end{align*}
	where $T_{\varphi_N}:\Omega_{N,A}\to\Omega_{N,A}, T_{\varphi_N}(\eta)=\eta\circ\varphi_N$.
\end{lemma}
\begin{proof}
	Since $\Omega_{N,A}$ is finite, it is sufficient to show that
	\begin{align*}
		(L_N(f\circ T_{\varphi_N}))=(L_N f)\circ T_{\varphi_N}
	\end{align*}
	for every $f:\Omega_{N,A}\to\bbR$.\\
	First, note that $UA=A+a$ implies
	\begin{align*}
		\varphi_N(y+A)=\varphi_N(y)+A+a.
	\end{align*}
	Consequently,
	\begin{align}\label{eqn:exclusion_symmetry}
		(y+A)\cap(z+A)\neq\emptyset\quad\Longleftrightarrow\quad (\varphi_N(y)+A)\cap(\varphi_N(z)+A)\neq\emptyset.
	\end{align}
	Moreover,
	\begin{align}\label{eqn:change_psi}
		(\eta\circ\varphi_N)^{x,y}=\eta^{\varphi_N(x),\varphi_N(y)}\circ\varphi_N\quad\forall\eta\in\Omega_{N,A}.
	\end{align}
	Using \eqref{eqn:exclusion_symmetry} and \eqref{eqn:change_psi} in the definition of the generator \eqref{eqn:gen_A} gives
	\begin{multline*}
		(L_N\,f)(T_{\varphi_N}(\eta))
		=\frac12\sum_{x\in\bbT_N^d}\sum_{|x-y|=1}\eta(\varphi_N(x))\prod_{z\neq x}\left(1-\eta(\varphi_N(z))\bbone_{(\varphi_N(y) + A) \cap (\varphi_N(z)+A) \neq \emptyset}\right)\\
		\times(f(\eta^{\varphi_N(x),\varphi_N(y)}\circ\varphi_N)-f(\eta\circ\varphi_N)).
	\end{multline*}
	The change of variables $x\mapsto \varphi_N^{-1}(x)$ and using that $\varphi_N^{-1}$ preserves nearest-neighbour bonds yields
	\begin{align*}
		(L_N\,f)(T_{\varphi_N}(\eta))
		&=\frac12\sum_{x\in\bbT_N^d}\sum_{|x-y|=1}\eta(x)\prod_{z\neq x}\left(1-\eta(z)\bbone_{(y + A) \cap (z+A) \neq \emptyset}\right)( f(T_{\varphi_N}(\eta^{x,y}))-f(T_{\varphi_N}(\eta)))\\
		&= (L_N\,(f\circ T_{\varphi_N}))(\eta).
	\end{align*}
	Since $P_t^N=e^{t L_N}$, it follows that for all $f:\Omega_{N,A}\to\bbR$ and $t\geq 0$
	\begin{align*}
		P_t^N(f\circ T_{\varphi_N}) = (P_t^N f)\circ T_{\varphi_N}.
	\end{align*}
	Therefore, for every probability measure $\mu^N$ and every function $f:\Omega_{N,A}\to\bbR$, we have
	\begin{align*}
		\int f\circ T_{\varphi_N}\diff (\mu^NP_t^N)
		&= \int P_t^N(f\circ T_{\varphi_N})\diff \mu^N
		= \int P_t^N(f)\circ T_{\varphi_N}\diff \mu^N\\
		&= \int P_t^N(f)\diff(\mu^N\circ T_{\varphi_N})
		= \int f\diff((\mu^N\circ T_{\varphi_N})P_t^N).
	\end{align*}
	This proves the claim.
\end{proof}
The next lemma relates symmetries of weak solutions of the hydrodynamic limit equation to algebraic constraints of the diffusion matrix.
It is sufficient to consider initial conditions that are small perturbations of constant density profiles.
The proof can be found in Appendix \ref{appendix:pde}.
\begin{lemma}\label{lemma:sol_to_matrix}
	Let Assumption \ref{ass:diffusion_tensor} hold.
	Let $\varphi:\bbT^2\to\bbT^2,x\mapsto Ux$ with $U\in\bbZ^{2\times2}$ orthogonal.
	Suppose that for every $\overline{\rho}\in(0,\rho_c^A)$, every $k\in\bbZ^2$, and every sufficiently small $\varepsilon>0$ such that
	\begin{align*}
		0<\rho_0^\varepsilon(x):=\overline{\rho}+\varepsilon\sin(2\pi k\cdot x)<\rho_c^A\quad \forall x\in\bbT^2,
	\end{align*}
	there exists a weak solution of \eqref{eqn:diff_eq} with initial condition $\rho_0^\varepsilon$ such that $(t,x)\mapsto\rho^\varepsilon(t,\varphi(x))$ is a weak solution to \eqref{eqn:diff_eq}, with initial condition $\rho_0^\varepsilon\circ\varphi$.
	Then
	\begin{align*}
		U^TD_A(\rho)U=D_A(\rho)\quad\forall \rho\in[0,\rho_c^A).
	\end{align*}
\end{lemma}
Using Lemma \ref{lemma:dist_rotated} and Lemma \ref{lemma:sol_to_matrix}, we show that the diffusion matrix $D_A$ is also invariant with respect to the action of $U$, if $A$ is invariant with respect to $U$ up to translation.
\begin{lemma}\label{lemma:rotated_particle}
	Let Assumptions \ref{ass:diffusion_tensor} and \ref{ass:hydrodynamic_limit} hold.
	Let $\varphi:\bbT^2\to\bbT^2, \varphi(x)= Ux$ with $U\in\bbZ^{2\times 2}$ orthogonal.
	If $UA=A+a$ for some $a\in\bbZ^2$, then
	\begin{align*}
		U^TD_A(\rho)U=D_A(\rho)\quad\forall \rho\in[0,\rho_c^A).
	\end{align*}
\end{lemma}
\begin{proof}
	For $N\in\bbN$, we denote 
	\begin{align*}
		\varphi_N:\bbT^2_N\to\bbT_N^2,\,x\mapsto Ux.
	\end{align*}
	Since $U\in\bbZ^{2\times 2}$ is orthogonal, we obtain $\varphi_N(\bbT^2_N)=\bbT^2_N$.\\
	Fix an admissible initial profile $\rho_0:\bbT^2\to[0,\rho_c^A)$ smooth, and let $(\mu^N)_N$ be a sequence of probability measures associated to $\rho_0$ and satisfying Assumption \ref{ass:hydrodynamic_limit}.
	We show that
	\begin{align}\label{eqn:conv_perturbed_initial}
		\{\pi_{N^2t}^N\circ\varphi^{-1}\,:\,t\in[0,T]\}\to\{\tilde{\rho}(t,x)\diff x\,:\,t\in[0,T]\}\quad\text{in distribution},
	\end{align}
	with $\tilde{\rho}$ a weak solution of \eqref{eqn:diff_eq} with initial condition $\rho_0\circ\varphi$.
	After that we show that
	\begin{align}\label{eqn:conv_perturbed_sol}
		\{\pi_{N^2t}^N\circ\varphi^{-1}\,:\,t\in[0,T]\}\to\{\rho(t,\varphi(x))\diff x\,:\,t\in[0,T]\}\quad\text{in distribution},
	\end{align}
	with $\rho$ a weak solution of \eqref{eqn:diff_eq} with initial condition $\rho_0$.
	Therefore, we have $\tilde{\rho}=\rho\circ\varphi$.
	The claim now follows from Lemma \ref{lemma:sol_to_matrix}.\\
	We first show that $\mu^N\circ T_{\varphi_N}$ is associated to $\rho_0\circ\varphi$. 
	Indeed, for every continuous $G:\bbT^2\to\bbR$ and every $\delta>0$,
	\begin{multline*}
		\lim_{N\to\infty}\mu^N\circ T_{\varphi_N}\left(\left|\int_{\bbT^2} G(x)\pi_0^N(\diff x)-\int_{\bbT^2} G(x)\rho_0(\varphi(x))\diff x\right|\geq \delta\right)\\
		=\lim_{N\to\infty}\mu^N\left(\left|\int_{\bbT^2} G(x)(\pi_0^N\circ\varphi^{-1})(\diff x)-\int_{\bbT^2} G(x)\rho_0(\varphi(x))\diff x\right|\geq \delta\right)\\
		=\lim_{N\to\infty}\mu^N\left(\left|\int_{\bbT^2} G(\varphi^{-1}(x))(\pi_0^N)(\diff x)-\int_{\bbT^2} G(\varphi^{-1}(x))\rho_0(x)\diff x\right|\geq \delta\right)=0.
	\end{multline*}
	Denote by $(\tilde{\pi}_t)_{t\geq 0}$ the empirical measure corresponding to the particle process with initial distribution $\mu^N\circ T_{\varphi_N}$.
	Assumption \ref{ass:hydrodynamic_limit} therefore implies that under the initial distribution $(\mu^N\circ T_{\varphi_N})$,
	\begin{align}\label{eqn:limit_equation}
		\{\tilde{\pi}_{N^2t}^N\,:\,t\in[0,T]\}\to\{\tilde{\rho}(t,x)\,\diff x\,:\,t\in[0,T]\}\quad\text{in distribution},
	\end{align}
	for $N\to\infty$, where $\tilde{\rho}(t,x)$ is a weak solution of \eqref{eqn:diff_eq} with initial condition $\rho_0\circ\varphi$.\\
	We now show \eqref{eqn:conv_perturbed_sol} using the continuous mapping theorem.
	By Lemma \ref{lemma:dist_rotated},
	\begin{align*}
		\{\tilde{\pi}_{N^2t}^N\,:\,t\in[0,T]\}\quad\text{ and }\quad\{\pi_{N^2t}^N\circ\varphi^{-1}\,:\,t\in[0,T]\}
	\end{align*}
	have the same law.\\
	The map 
	\begin{align*}
		\mathcal{M}_+(\bbT^2)\to\mathcal{M}_+(\bbT^2),\,\mu\mapsto \mu\circ\varphi^{-1},
	\end{align*}
	is continuous with respect to the weak topology.
	Hence, the induced map on $D([0,T],\mathcal{M}_+(\bbT^2))$ is continuous.
	Applying the continuous mapping theorem gives for $N\to\infty$
	\begin{align}\label{eqn:limit_action}
		\{\pi_{N^2t}^N\circ\varphi^{-1}\,:\,t\in[0,T]\}\to\{\rho(t,\varphi(x))\,\diff x\,:\,t\in[0,T]\}\quad\text{in distribution},
	\end{align}
	Combining \eqref{eqn:limit_equation} and \eqref{eqn:limit_action}, and using uniqueness of the limit in distribution, we obtain $\tilde{\rho}=\rho\circ\varphi$.\\
	By Lemma \ref{lemma:sol_to_matrix}, the claim follows.
\end{proof}
The next lemma relates the algebraic conditions of the diffusion matrix $D_A$ to the entries of $D_A$ if the process is symmetric with respect to reflection across the $x$-axis and rotation by $\frac\pi2$.
\begin{lemma}\label{lemma:Rotation}
	Let
	\begin{align*}
		U_1=
		\begin{pmatrix}
			1&0\\
			0&-1
		\end{pmatrix},\quad
		U_2=
		\begin{pmatrix}
			0&-1\\
			1&0
		\end{pmatrix},
	\end{align*}
	be the orthogonal matrices that correspond to the reflection across the $x$-axis and rotation by $\frac{\pi}2$ respectively.
	Let $M\in\bbR^{2\times 2}$.
	If $U_1MU_1^T=M$ then $M$ is a diagonal matrix.
	If additionally $U_2MU_2^T=M$ then there exists $d\in\bbR$ such that $M=dI$.
\end{lemma}
\begin{proof}
	We can write $U_1MU_1^T=M$ as $MU_1^T=U_1^TM$.
	Hence $M$ commutes with $U_1^T$ and $U_1^T$ is diagonal with distinct values at the diagonal and therefore $M$ needs to be diagonal (see \cite[Theorem 1.37]{Higham08}).\\
	If in addition $U_2MU_2^T=M$ we obtain
	\begin{align*}
		U_2MU_2^T=
		\begin{pmatrix}
			M_{2,2}&0\\
			0 & M_{1,1}
		\end{pmatrix}
		= M.
	\end{align*}
	Thus, there exists a $d\in\bbR$ such that $M=dI$.
\end{proof}
We can now state that the diffusion matrix is diagonal if $A$ is rectangular and isotropic if $A$ is a square.
\begin{cor}\label{cor:diffusion_form}
	Let Assumptions \ref{ass:diffusion_tensor} and \ref{ass:hydrodynamic_limit} hold.
	Let $A=\{0,\dots,r_1-1\}\times\{0,\dots,r_2-1\}$ for $r_1,r_2\geq 1$.
	Then $D_A:[0,\rho_c^A)\to \bbR^{2\times2}$ is a diagonal matrix.
	If in addition $r_1=r_2$ then there exists $d_A:[0,\rho_c^A)\to\bbR$ such that $D_A(\rho) = d_A(\rho)I$ for all $\rho\in[0,\rho_c^A)$.
\end{cor}
\begin{proof}
	Let 
	\begin{align*}
		U_1 = 
		\begin{pmatrix}
			1&0\\
			0&-1
		\end{pmatrix}.
	\end{align*}
	Then $U_1A=A-(r_2-1)
	\begin{pmatrix}
		0\\1
	\end{pmatrix}$.
	Hence $D_A(\rho)$ is diagonal for all $\rho\in[0,\rho_c^A)$ by Lemma \ref{lemma:rotated_particle} and \ref{lemma:Rotation}.\\
	Let $r_1=r_2$ and 
	\begin{align*}
		U_2=\begin{pmatrix}
			0&-1\\
			1&0
		\end{pmatrix}.
	\end{align*}
	Then $U_2A=A-(r_1-1)
	\begin{pmatrix}
		1\\0
	\end{pmatrix}$.
	Hence by Lemma \ref{lemma:rotated_particle} and \ref{lemma:Rotation}, there exists $d_A:[0,\rho_c^A)\to\bbR$ with $D_A(\rho)=d_A(\rho)I$ for all $\rho\in[0,\rho_c^A)$.
\end{proof}
\subsection{Estimation of the diffusion matrix}
We use a Fourier mode perturbation of a constant state as an initial condition to numerically estimate the diffusion matrix $D_A(\rho)$ for some rectangular $A\subset\bbZ^2$.
By Corollary \ref{cor:diffusion_form}, we may write
\begin{align*}
	D_A(\rho)=
	\begin{pmatrix}
		D_{1,1}(\rho)&0\\
		0&D_{2,2}(\rho)
	\end{pmatrix}.
\end{align*}
Notice that for a linear diffusion problem 
\begin{align*}
	\begin{cases}
	\begin{aligned}
		\partial_t \psi = \operatorname{div}(M\nabla \psi)\quad&&\text{on }(0,T)\times\bbT^2,\\
		\psi(0,x_1,x_2)=\sin(2\pi x_1)&&\quad \forall (x_1,x_2)\in\bbT^2,
	\end{aligned}	
	\end{cases}
\end{align*}
with $M\in\bbR^{2\times 2}$ symmetric positive definite, the unique solution is given by
\begin{align*}
	\psi(t,x_1,x_2)=\exp(-4\pi^2 M_{1,1}t)\sin(2\pi x_1)
\end{align*}
Therefore, we can use the solution $\psi$ to get the value $M_{1,1}$ and by rotation also the value $M_{2,2}$.
We want to use this to obtain approximations of $D_A(\rho)$ for $\rho\in[0,\rho_c^A)$.
However, for a nonlinear diffusion equation of the form \eqref{eqn:diff_eq} the preceding arguments does not apply directly.
The next lemma states that for a sequence of suitable initial conditions the corresponding solutions to \eqref{eqn:diff_eq} approximate the solution of a linear diffusion problem.
Its proof can be found in Appendix \ref{appendix:pde}.
\begin{lemma}\label{lemma:conv_lin_diff}
	Let Assumption \ref{ass:diffusion_tensor} hold.
	Let $\overline{\rho}\in(0,\rho_c^A)$, $\psi_0\in L^\infty(\bbT^2)$ and $T>0$.
	For each sufficiently small $\varepsilon>0$, let $\rho^\varepsilon$ be a weak solution to \eqref{eqn:diff_eq} with initial condition $\rho_0^\varepsilon:=\overline{\rho}+\varepsilon\psi_0$.
	Then 
	\begin{align*}
		\frac{\rho^\varepsilon-\overline{\rho}}{\varepsilon}\to\psi \quad\text{for }\varepsilon\to 0\quad\text{in } L^\infty(0,T;L^2(\bbT^2)),
	\end{align*}
	with $\psi$ the weak solution of
	\begin{align}\label{eqn:lin_orig}
		\begin{cases}
			\begin{aligned}
				\partial_t \psi&=\operatorname{div}(D_A(\overline{\rho})\nabla\psi) \quad&&\text{on }(0,T]\times\bbT^2\\
				\psi(0,\cdot)&=\psi_0(\cdot) &&\text{on }\bbT^2.
			\end{aligned}
		\end{cases}
	\end{align}
\end{lemma}
Based on Lemma \ref{lemma:conv_lin_diff} we can state the strategy how we obtain approximations of $D_A$.
For this we fix $ \overline{\rho}\in(0,\rho_c^A)$ and consider the initial condition
\begin{align}\label{eqn:estimate_initial}
	\rho_0^\varepsilon(x_1,x_2) := \overline{\rho}+\varepsilon\sin(2\pi x_1),
\end{align}
where $\varepsilon>0$ is sufficiently small such that $0\leq \rho_0^\varepsilon<\rho_c^A$.
Denote by $\rho^\varepsilon$ a weak solution of \eqref{eqn:diff_eq} with initial condition \eqref{eqn:estimate_initial}.
By Lemma \ref{lemma:conv_lin_diff}, we can split $\rho^\varepsilon$ into
\begin{align*}
	\rho^\varepsilon(t,x_1,x_2) = \overline{\rho}+\varepsilon\psi(t,x_1,x_2)+r^\varepsilon(t,x_1,x_2)
\end{align*}
with $\frac1{\varepsilon}\|r^\varepsilon\|_{L^\infty(0,T;L^2(\bbT^2))}\to 0$ for $\varepsilon\to0$ and $\psi(t,x_1,x_2):=\exp(-4\pi^2 D_{1,1}(\overline{\rho}) t)\sin(2\pi x_1)$ the weak solution of \eqref{eqn:lin_orig} with $\psi_0(x_1,x_2)=\sin(2\pi x_1)$.\\
We define the normalized first sine-mode coefficient by
\begin{align*}
	b_1^\varepsilon(t):=\frac2{\varepsilon}\int_{\bbT^2}\rho^\varepsilon(t,x_1,x_2)\sin(2\pi x_1)\diff(x_1,x_2).
\end{align*}
For the first sine-mode $b_1^\varepsilon(t)$ at time $t\in[0,T]$, we have
\begin{multline}\label{eqn:fourier_coeff}
	\frac\varepsilon{2} b_1^\varepsilon(t) = \int_{\bbT^2}\rho^\varepsilon(t,x_1,x_2)\sin(2\pi x_1)\diff (x_1,x_2)
=\underbrace{\int_{\bbT^2} \overline{\rho}\sin(2\pi x_1)\diff(x_1,x_2)}_{=:I_1}\\
+\underbrace{ \int_{\bbT^2} \varepsilon\psi(t,x_1,x_2)\sin(2\pi x_1)\diff(x_1,x_2)}_{=:I_2}
	+\underbrace{ \int_{\bbT^2} r^\varepsilon(t,x_1,x_2)\sin(2\pi x_1)\diff(x_1,x_2)}_{=:I_3}.
\end{multline}
The first part $I_1$ vanishes.
From the solution representation of $\psi$, it follows that
\begin{align*}
	I_2 = \frac\varepsilon2\exp(-4\pi^2 D_{1,1}(\overline{\rho})t).
\end{align*}
Using H\"older's inequality and Lemma \ref{lemma:conv_lin_diff} yields
\begin{align*}
	|I_3|\leq \|r^\varepsilon\|_{L^\infty(0,T;L^2(\bbT^2))}\frac1{\sqrt{2}}= o(\varepsilon).
\end{align*}
Dividing both sides of \eqref{eqn:fourier_coeff} by $\frac{\varepsilon}2$ we obtain uniformly for $t\in[0,T]$ as $\varepsilon\to 0$
\begin{align*}
	b_1^\varepsilon(t) = \exp(-4\pi^2D_{1,1}(\overline{\rho})t)+o(1).
\end{align*}
Hence, for $\varepsilon>0$ small enough we use
\begin{align}\label{eqn:ode_b1}
	b_1^\varepsilon(t) \approx \exp(-4\pi^2D_{1,1}(\overline{\rho})t),
\end{align}
to approximate the diffusion coefficient.\\
When simulating the particle process on a finite mesh, we obtain noisy measurements of $b_1^\varepsilon$ due to $\varepsilon>0$ and stochastic effects.
To reduce the noise, we consider $m$ different instances of particle processes $((\pi_t^{N,j})_{t\leq T})_{j=1}^m$ with distribution associated to $\rho_0^\varepsilon\diff x$ with $\rho_0^\varepsilon$ given by \eqref{eqn:estimate_initial}.
We approximate $b_1^\varepsilon$ at the time steps $0=t_0<\dots<t_n=T$ by
\begin{align*}
	b_1^\varepsilon(t_i)\approx \tilde{b}_1(t_i) := \frac2{\varepsilon m}\sum_{j=1}^m\langle \pi_{t_i}^{N,j},\sin(2\pi\cdot)\rangle\quad i=0,\dots,n.
\end{align*}
Notice that taking the $\log$ of both sides of \eqref{eqn:ode_b1} we obtain that $t\mapsto\log(b_1^\varepsilon(t))$ is approximately an affine linear function with slope $-4\pi^2 D_{1,1}(\overline{\rho})$.
We utilize this by using a least-squares fit of $(t_i,\log(\tilde{b}_1(t_i)))_{i=0,\dots,n}$ to estimate $D_{1,1}(\overline{\rho})$.\\
The coefficient $D_{2,2}(\overline{\rho})$ is estimated analogously by using the initial perturbation
\begin{align*}
	\rho_0^\varepsilon(x_1,x_2) = \overline{\rho}+\varepsilon\sin(2\pi x_2)
\end{align*}
and measuring the decay of the first sine mode in $x_2$-direction.\\
We conduct experiments for $A=\{0,\dots,r\}^2$ for $r=1,2,3$ and $A=\{0,1,2\}\times\{0,1\}$.
To measure the diffusion coefficient for $A$ we average the measured $\tilde{b}_1(t)$ over $480$ runs on a lattice of size $N=1024$ and initial condition given by \eqref{eqn:estimate_initial} with $\varepsilon=0.005$ and $T=0.01$.
In our observation, making $\varepsilon$ or $T$ smaller and $N$ larger has no influence on the measured diffusion coefficients.
In all our experiments, the relative discrete $L^2$ error between the data $(\log(\tilde{b}_1(t_i)))_{i=0,\dots,n}$ and the corresponding linear least-squares fit is below $0.02$.
The computation time is primarily determined by the number of particles, since most of the computational cost is spent sampling the random numbers.
Sampling the random numbers needed to attempt a jump is much more costly than checking that the new configuration is admissible.
\\
\begin{figure} 
\caption{Diffusion coefficients as functions of the density} \label{fig:results}
	\begin{subfigure}[b]{0.5\textwidth}
		\includegraphics[width=1\textwidth]{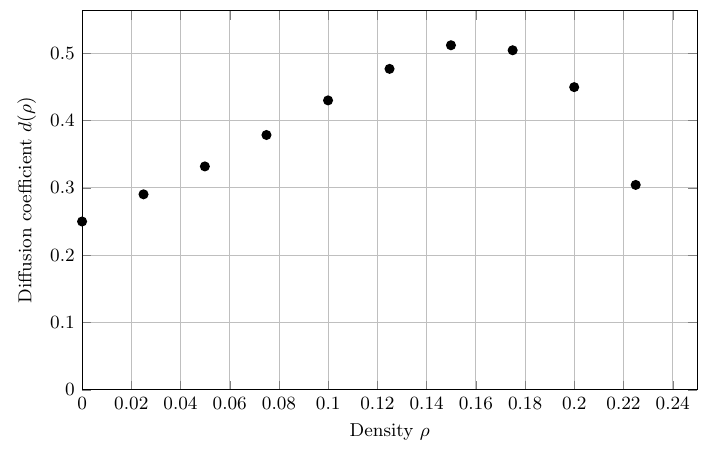}
		\caption{Measurement for $A=\{0,1\}^2$}
	\end{subfigure}%
	\begin{subfigure}[b]{0.5\textwidth}
		\includegraphics[width=1\textwidth]{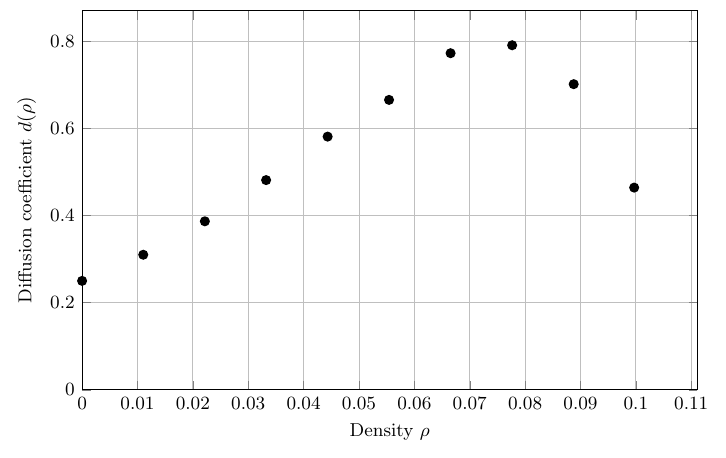}
		\caption{Measurement for $A=\{0,1,2\}^2$}
	\end{subfigure}
	\begin{subfigure}[b]{0.5\textwidth}
		\includegraphics[width=1\textwidth]{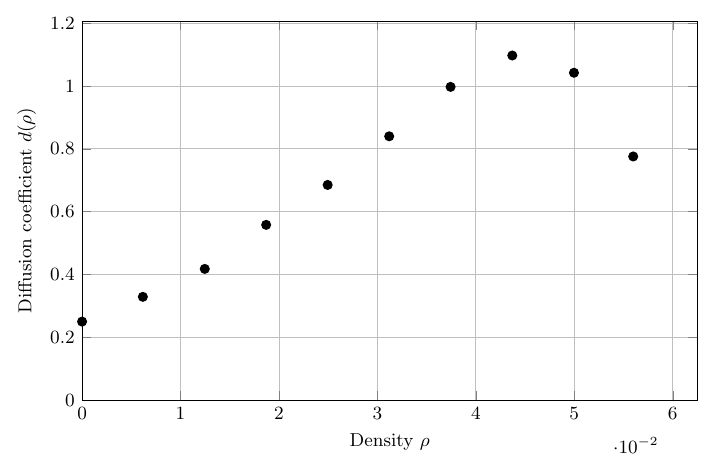}
		\caption{Measurement for $A=\{0,1,2,3\}^2$}
	\end{subfigure}%
	\begin{subfigure}[b]{0.5\textwidth}
		\includegraphics[width=1\textwidth]{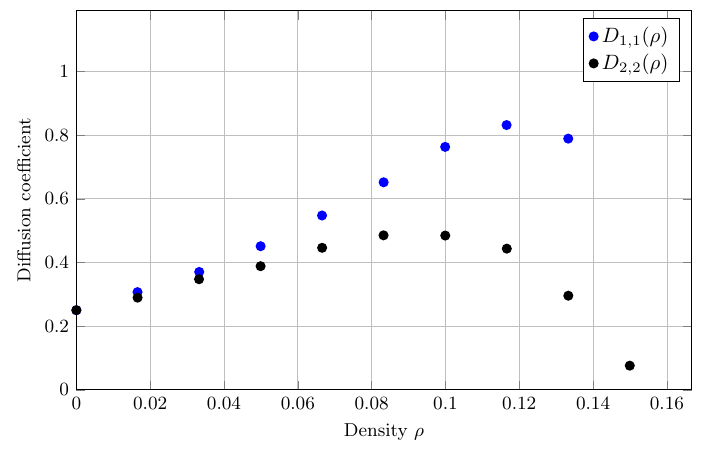}
		\caption{Measurement for $A=\{0,1,2\}\times\{0,1\}$}
	\end{subfigure}
\end{figure}
The numerical results show a clear dependence of the diffusion coefficient on both the density and the geometry of the exclusion zone.
For the square exclusion zones, Corollary \ref{cor:diffusion_form} implies that the hydrodynamic limit equation is isotropic.
In all three square cases and small densities, the process behaves similarly to the one dimensional process that has a monotonically increasing diffusion coefficient given by (see \eqref{PDE})
\begin{align*}
	d_r(\rho) = \frac1{2((r-1)\rho-1)^2}\quad\forall \rho\in\left[0,\frac1r\right],
\end{align*}
for an exclusion zone given by $\{0,\dots,r-1\}$.
However, $d_A(\rho)$ in the two dimensional case reaches a maximum at an intermediate density and decreases from there on.
There is a good heuristic reason for this: for low densities, motion in the two coordinate directions is largely independent, and 
the system resembles the one-dimensional one. Namely, that higher density increases the diffusivity. 
However, when getting closer to the critical density of the phase transition, another effect appears: more often, the particles cannot 
move because their motion is blocked by a particle which is not in the same row/column, but in the neighbouring one. This reduces the 
total mobility. 

For the rectangular exclusion zone $A=\{0,1,2\}\times\{0,1\}$, Corollary \ref{cor:diffusion_form} implies that the hydrodynamic limit equation has a diagonal diffusion matrix.
For both diagonal entries we see an initial increase with the density and both entries are close to each other.
For a density $\rho\approx 0.09$, $D_{22}(\rho)$ reaches its maximum and decreases from there on, while $D_{11}(\rho)$ keeps increasing until $
\rho\approx 0.12$ and decreases from there. Note here that the mobility is actually {\em higher} in the direction where the rectangle has 
its longer side. This is clear from the intuition we gained from the square case: on the one hand, longer shapes in the direction of movement
actually increase the mobility, but longer shapes in the opposite direction act as obstructions and decrease it. 

\begin{bem} We end this paper by discussing the discrepancy of our results with the predictions of \cite{Bio11}, and 
a discussion of the likely reasons. \cite{Bio11} treats the case $A=\{0\} \times \{0,\dots,r-1\}$, but their method can be easily modified 
to obtain both the one-dimensional case, and the case where $A$ is a square. In the former case, the prediction is a diffusivity proportional
to $\rho^{r-1}$, which is correct only in the case $r=1$. In the two-dimensional square case, the prediction is 
\begin{align}\label{eqn:bio_diff}
	D(\rho) = {\rm const} (r^2\rho(1-r^2\rho))^{r-1} I.
\end{align}
While this expression does capture the non-monotonicity of the diffusivity, its maximum lies at precisely half the maximal density, which clearly
does not match our numerical results. More seriously, this expression (as well as the one-dimensional one) predicts diffusivity going to zero
as the density vanishes, which is physically unreasonable: at low densities, the excluded volume constraint becomes irrelevant, and the diffusivity
should thus approach a constant, which is also confirmed by our data. 

The likely reason for this discrepancy is actually mentioned by the authors of \cite{Bio11} themselves. As they note towards the end of Section D, 
their derivation assumes that the stationary measure of the process is product. This is actually very far from being correct, and leads to the 
observed, rather drastic, discrepancies. In conclusion, deriving PDE as diffusive limits from particle systems is a delicate matter, and even 
reasonable heuristic arguments can easily lead to the wrong effective equation. In particular, beyond the numerics that we presented and that 
we believe is trustworthy, we have no suggestion of a formula for the diffusivity in the case $A = \{0,1\}^2$ or similar ones. 
\end{bem}

\section*{Acknowledgements}
The authors gratefully acknowledge the computing time provided to them on the high-performance computer Lichtenberg II at TU Darmstadt, funded by the German Federal Ministry of Research, Technology and Space (BMFTR) and the State of Hesse.
\newpage

\begin{appendix}
	\section{Results for nonlinear diffusion equations}\label{appendix:pde}
\begin{bem}\label{bem:new_weak_form}
	Let Assumption \ref{ass:diffusion_tensor} hold and let $\rho\in L^2(0,T;H^1(\bbT^2))\cap L^\infty(0,T;L^2(\bbT^2))$ be a weak solution of \eqref{eqn:diff_eq}.
	We will show that $\rho\in H^1(0,T;H^{-1}(\bbT^2))$ and derive a more suitable equivalent weak formulation.
	To this end, we define
	\begin{align*}
		F:[0,T]\times H^1(\bbT^2)\to\bbR,(t,G)\mapsto\int_{\bbT^2} D_A(\rho(t,x))\nabla\rho(t,x)\cdot\nabla G(x)\diff x.
	\end{align*}
	For all $G\in H^1(\bbT^2)$ and $t\in[0,T]$, we have
	\begin{align*}
		|F(t,G)|\leq \Lambda\|\nabla\rho(t,\cdot)\|_{L^2(\bbT^2)}\|G\|_{H^1(\bbT^2)}.
	\end{align*}
	Hence $F\in L^2(0,T;H^{-1}(\bbT^2))$ and for all $G\in H^1(\bbT^2)$, we have
	\begin{align*}
		\int_{\bbT^2}\rho(t,x)G(x)\diff x=\int_{\bbT^2}\rho_0(x)G(x)\diff x-\int_0^tF(s,G)\diff s
	\end{align*}
	Therefore, $\rho(t,\cdot) = \rho_0-\int_0^t F(s,\cdot)\diff s$ in $H^{-1}(\bbT^2)$ and \cite[Proposition 1.1 in III.1]{showalter2013} therefore implies that $\rho\in H^1(0,T;H^{-1}(\bbT^2))$ with $\partial_t\rho=-F$.
	In particular, the notion of weak solution of \eqref{eqn:diff_eq} is equivalent to the following identity 
	\begin{align}\label{eqn:weak_form_diff}
		\langle\partial_t\rho(t,\cdot),G(\cdot)\rangle_{H^{-1},H^1}=-\int_{\bbT^2} D_A(\rho(t,x))\nabla \rho(t,x)\cdot \nabla G(x)\diff x\quad\forall G\in H^1(\bbT^2),
	\end{align}
	which holds for a.e. $t\in[0,T]$.
	We denote $\langle\cdot,\cdot\rangle_{H^{-1},H^1}$ the duality pairing.
\end{bem}
With the weak formulation \eqref{eqn:weak_form_diff} from Remark \ref{bem:new_weak_form} we obtain the following weak maximum principle.
The maximum principle is necessary in order to prove Lemma \ref{lemma:conv_lin_diff}.
\begin{lemma}\label{lemma:max_princ}
	Let $\rho_0\in L^\infty(\bbT^2)$.
	Then the weak solution $\rho\in L^2(0,T;H^1(\bbT^2))$ of \eqref{eqn:diff_eq} satisfies
	\begin{align*}
		\|\rho-c\|_{L^\infty([0,T]\times\bbT^2)}\leq \|\rho_0-c\|_{L^\infty(\bbT^2)}
	\end{align*}
	for every $c\in\bbR$.
\end{lemma}
\begin{proof}
	Testing \eqref{eqn:weak_form_diff} with $(\rho(t,\cdot)-c-M)_+:=\max\{0,\rho(t,\cdot)-c-M\}$ for $t\in[0,T]$ yields
	\begin{align}\label{eqn:test_cutoff}
		\frac12\partial_t\|(\rho(t,\cdot)-c-M)_+\|_{L^2(\bbT^2)}^2
		&=\left\langle\partial_t\rho(t,x),(\rho(t,x)-c-M)_+\right\rangle_{H^{-1},H^1}\\
		&= -\int_{\{\rho-c>M\}} D_A(\rho(t,x))\nabla\rho(t,x)\cdot\nabla\rho(t,x)\diff x
		\leq 0.
	\end{align}
	Choosing $M> \|\rho_0-c\|_{L^\infty(\bbT^2)}$ in \eqref{eqn:test_cutoff} yields for all $t\in[0,T]$
	\begin{align}\label{eqn:lower_bound}
		\|(\rho(t,\cdot)-c-M)_+\|_{L^2(\bbT^2)}\leq \|(\rho_0-c-M)_+\|_{L^2(\bbT^2)}=0.
	\end{align}
	Similarly, testing \eqref{eqn:weak_form_diff} with $(c-M-\rho)_+$ we obtain
	\begin{align}\label{eqn:upper_bound}
		\|(c-M-\rho(t,\cdot))_+\|_{L^2(\bbT^2)}\leq \|(c-M-\rho_0)_+\|_{L^2(\bbT^2)}=0.
	\end{align}
	Combining \eqref{eqn:lower_bound} and \eqref{eqn:upper_bound} yields for a.e. $(t,x)\in[0,T]\times\bbT^2$
	\begin{align}\label{eqn:rho_bounds}
		\begin{split}
			(\rho(t,x)-c)-M&\leq 0\\
			c-M-\rho(t,x) &\leq 0 
		\end{split}
		\implies -M\leq \rho(t,x)-c\leq M.
	\end{align}
	Taking the limit $M\to\|\rho_0-c\|_{L^\infty(\bbT^2)}$ in \eqref{eqn:rho_bounds} yields the claim.
\end{proof}
\begin{proof}[Proof of Lemma \ref{lemma:conv_lin_diff}]
	Write $\rho^\varepsilon(t,x) = \overline{\rho}+\varepsilon\psi(t,x) +r^\varepsilon(t,x)$.
	It is enough to prove $\frac1{\varepsilon}\|r^\varepsilon\|_{L^\infty(0,T;L^2(\bbT^2))}\to 0$ for $\varepsilon \to 0$.\\
	We derive an equation for $r^\varepsilon$ and prove the claim by energy estimates for $r^\varepsilon$.
	Subtracting the weak formulation of
	\begin{align*}
		\partial_t(\varepsilon\psi)=\operatorname{div}(D_A(\overline{\rho})\nabla(\varepsilon\psi))
	\end{align*}
	from the weak formulation for $\rho^\varepsilon$ we obtain for a.e. $t\in[0,T]$ and $G\in L^2(0,T;H^1(\bbT^2))$
	\begin{align}\label{eqn:weak_form_residual}
		\begin{split}
			\langle \partial_t r^\varepsilon(t,\cdot),G(t,\cdot)\rangle_{H^{-1},H^1}
			&=- \int_{\bbT^2} \left[D_A(\rho^\varepsilon(t,x))\nabla\rho^\varepsilon(t,x)-\varepsilon D_A(\overline{\rho})\nabla\psi(t,x)\right]\cdot\nabla G(t,x)\diff x\\
			&=- \int_{\bbT^2} \left[D_A(\rho^\varepsilon(t,x))\nabla r^\varepsilon(t,x)+\varepsilon (D_A(\rho^\varepsilon(t,x))-D_A(\overline{\rho}))\nabla\psi(t,x)\right]\cdot\nabla G(t,x)\diff x.
		\end{split}
	\end{align}
	From $\rho^\varepsilon,\psi\in L^2(0,T;H^1(\bbT^2))\cap H^1(0,T;H^{-1}(\bbT^2))$, it follows that $r^\varepsilon\in L^2(0,T;H^1(\bbT^2))\cap H^1(0,T;H^{-1}(\bbT^2))$ and we can take $G(x)=r^\varepsilon(t,x)$ as a test function in \eqref{eqn:weak_form_residual}.
	Using H\"older's inequality and Young's inequality on the right hand side, we obtain
	\begin{align*}
		\frac12\|r^\varepsilon(t,\cdot)\|_{L^2(\bbT^2)}^2+\frac\lambda2\|\nabla r^\varepsilon\|_{L^2([0,t]\times\bbT^2)}^2\leq \frac{\varepsilon^2}{2\lambda}\|D_A(\rho^\varepsilon)-D_A(\overline{\rho})\|_{L^\infty([0,t]\times\bbT^2)}^2\|\nabla\psi\|_{L^2([0,t]\times\bbT^2)}^2,
	\end{align*}
	where we used $r^\varepsilon(0,\cdot)=0$.
	Lemma \ref{lemma:max_princ} implies
	\begin{align*}
		\|\rho^\varepsilon-\overline{\rho}\|_{L^\infty([0,T]\times\bbT^2)}
		\leq \varepsilon\|\psi_0\|_{L^\infty(\bbT^2)}.
	\end{align*}
	Thus, $\rho^\varepsilon$ converges uniformly to $\overline{\rho}$ for $\varepsilon\to 0$.
	It follows from the continuity of $D_A$ that
	\begin{align*}
		\|D_A(\rho^\varepsilon)-D_A(\overline{\rho})\|_{L^\infty([0,T]\times\bbT^2)}\to 0\quad\text{for }\varepsilon \to 0.
	\end{align*}
	Therefore
	\begin{align*}
		\left\|\frac{\rho^\varepsilon-\overline{\rho}}{\varepsilon}-\psi\right\|_{L^{\infty}(0,T;L^2(\bbT^2))}
		= \frac1\varepsilon\|r^\varepsilon\|_{L^{\infty}(0,T;L^2(\bbT^2))}\to 0\quad\text{ for }\varepsilon \to 0.
	\end{align*}
\end{proof}
\begin{proof}[Proof of Lemma \ref{lemma:sol_to_matrix}]
	We first prove that $\rho^\varepsilon$ also satisfies the weak formulation
	\begin{multline}\label{eqn:weak_formulation_after_transform}
		\int_{\bbT^2}\rho^\varepsilon(t,x)G(x)\diff x-\int_{\bbT^2}\rho_0^\varepsilon(x)G(x)\diff x
		=-\int_0^t\int_{\bbT^2}UD_A(\rho^\varepsilon(s,x))U^T\nabla\rho^\varepsilon(s,x)\cdot\nabla G(x)\diff x\diff s,
	\end{multline}
	for all $G:\bbT^2\to\bbR$ smooth.
	After that we use Lemma \ref{lemma:conv_lin_diff} and the uniqueness for the linearized problem to prove the claim.\\
	To prove \eqref{eqn:weak_formulation_after_transform} let $G:\bbT^2\to\bbR$ be smooth.
	Then $G\circ\varphi$ is smooth and using $G\circ\varphi$ as a test function in the weak formulation for $(t,x)\mapsto\rho^\varepsilon(t,\varphi(x))$ yields
	\begin{multline}
		\int_{\bbT^2}\rho^\varepsilon(t,\varphi(x))G(\varphi(x))\diff x-\int_{\bbT^2}\rho_0^\varepsilon(\varphi(x))G(\varphi(x))\diff x\\
		=-\int_0^t\int_{\bbT^2}D_A(\rho^\varepsilon(s,\varphi(x)))U^T\nabla\rho^\varepsilon(s,\varphi(x))\cdot(U^T\nabla G(\varphi(x)))\diff x\diff s\\
		=-\int_0^t\int_{\bbT^2}UD_A(\rho^\varepsilon(s,\varphi(x)))U^T\nabla\rho^\varepsilon(s,\varphi(x))\cdot\nabla G(\varphi(x))\diff x\diff s.
	\end{multline}
	After the change of variables $v=\varphi(x)$, it follows that $\rho^\varepsilon$ also satisfies \eqref{eqn:weak_formulation_after_transform}.\\
	By Lemma \ref{lemma:conv_lin_diff},
	\begin{align*}
		\frac{\rho^{\varepsilon}-\overline{\rho}}{\varepsilon}\to \psi\quad\text{for }\varepsilon\to 0,
	\end{align*}
	where $\psi$ is the unique solution of the linear diffusion equation \eqref{eqn:lin_orig}.
	On the other hand, applying Lemma \ref{lemma:conv_lin_diff} to \eqref{eqn:weak_formulation_after_transform} shows that the same function $\psi$ is the unique solution of
	\begin{equation}\label{eqn:lin_trans}
		\begin{aligned}
		\partial_t\psi &= \operatorname{div}(UD_A(\overline{\rho})U^T\nabla\psi)\quad &&\text{on }(0,T]\times\bbT^2\\
		\psi(0,x) &= \sin(2\pi k\cdot x)\quad&&\text{for all }x\in\bbT^2.
		\end{aligned}
	\end{equation}
	The solution of \eqref{eqn:lin_orig} with initial condition $\psi_0(x) =\sin(2\pi k\cdot x)$ is
	\begin{align*}
		\psi(t,x)=\exp(-4\pi^2 tk^TD_A(\overline{\rho})k)\sin(2\pi k\cdot x),
	\end{align*}
	whereas the solution of \eqref{eqn:lin_trans} is
	\begin{align*}
		\psi(t,x)=\exp(-4\pi^2 tk^TUD_A(\overline{\rho})U^Tk)\sin(2\pi k\cdot x)
	\end{align*}
	Since the two solutions coincide and $\overline{\rho}\in(0,\rho_c^A)$ was arbitrary, we obtain
	\begin{align}\label{eqn:quadratic_form}
		k^TD_A(\overline{\rho})k=k^TUD_A(\overline{\rho})U^Tk\quad\forall k\in\bbZ^2,\,\overline{\rho}\in(0,\rho_c^A).
	\end{align}
	Taking $k=(1,0)^T, k=(0,1)^T$ and $k=(1,1)^T$ in \eqref{eqn:quadratic_form} yields
	\begin{align}\label{eqn:testet_form}
		\begin{split}
		(D_A(\overline{\rho}))_{1,1} &= (UD_A(\overline{\rho})U^T)_{1,1},\\
		(D_A(\overline{\rho}))_{2,2} &= (UD_A(\overline{\rho})U^T)_{2,2},\\
		\sum_{i,j=1}^2(D_A(\overline{\rho}))_{i,j} &= \sum_{i,j=1}^2(UD_A(\overline{\rho})U^T)_{i,j}
		\end{split}
	\end{align}
	respectively.
	It follows from the first two rows of \eqref{eqn:testet_form} that $D_A(\overline{\rho})$ and $UD_A(\overline{\rho})U^T$ have identical diagonal elements.
	Together with the last row of \eqref{eqn:testet_form} and the symmetry of $D_A(\overline{\rho})$ and $UD_A(\overline{\rho})U^T$ this implies that $D_A(\rho)=UD_A(\rho)U^T$ for $\rho\in(0,\rho_c^A)$.
	The equality for $\rho=0$ follows from the continuity of $D_A$.
\end{proof}
\end{appendix}

\bibliographystyle{plain} 
\bibliography{bib}

\end{document}